\documentclass[11pt]{amsart}
\usepackage{url,hyperref}
\usepackage{caption}
\usepackage{fourier}
\usepackage{tikz}
\usepackage{pgfplots}
\pgfplotsset{compat=1.5}
\usepackage{mathtools}
\usepackage{enumerate}
\usepackage[margin=2.5cm]{geometry}
\DeclareMathOperator{\Ap}{Ap}
\DeclarePairedDelimiter{\ceil}{\lceil}{\rceil}
\DeclarePairedDelimiter{\floor}{\lfloor}{\rfloor}

\newtheorem{theorem}{Theorem}

\newtheorem{proposition}[theorem]{Proposition}
\newtheorem{corollary}[theorem]{Corollary}
\newtheorem{lemma}[theorem]{Lemma}
\theoremstyle{remark}
\newtheorem{remark}[theorem]{Remark}

\title{The tree of numerical semigroups with low multiplicity}
\date{}
\author[Garc\'{\i}a-S\'anchez]{P. A. Garc\'ia-S\'anchez}
\author[Mar\'in-Arag\'on]{D. Mar\'in-Arag\'on}
\author[Robles-P\'erez]{A. M. Robles-P\'erez}

\begin{document}

\maketitle

\begin{abstract}
    We show that the number of numerical semigroups with multiplicity three, four or five and fixed genus is increasing as a function in the genus. To this end we use the Kunz polytope for these multiplicities. Counting numerical semigroups with fixed multiplicity and genus is then an integer partition problem with some extra conditions (those of membership to the Kunz polytope). For the particular case of multiplicity four, we are able to prove that the number of numerical semigroups with multiplicity four and genus $g$ is the number of partitions  $x+y+z=g+6$ with $0<x\le y\le z$, $x\neq 1$, $y\neq 2$ and $z\neq 3$.
\end{abstract}

\section*{Introduction}

Let $\mathbb{N}$ be the set of the nonnegative integers. A subset $S$ of $\mathbb{N}$ is a \emph{numerical semigroup} if for every $a,b\in S$, $a+b\in S$, $0\in S$ and $\mathbb{N}\setminus S$ is a finite set. The cardinality of the set $N\setminus S$ is called the \emph{genus} of $S$, denoted by $g(S)$, and its maximum is called the \emph{Frobenius number} of $S$ and it is denoted by $\mathrm{F}(S)$. Sometimes the Frobenius number is replaced by the \emph{conductor} of $S$, which is the least integer $\mathrm{c}(S)$ such that $\mathrm{c}(S)+n\in S$ for all $n\in {\mathbb N}$. In fact, $\mathrm{F}(S)=\mathrm{c}(S)-1$.

Let $A$ be a  subset of $\mathbb{N}$, we denote by $\langle A \rangle$ the least monoid containing $A$. It is easy to show that  $\langle A \rangle$ is a numerical semigroup if and only if $\gcd(A)=1$. If $S$ is a numerical semigroup and $\langle A\rangle =S$, then we say that $A$ is a generating system for $S$. If no proper subset of $A$ generates $S$ as a monoid, then $A$ is a \emph{minimal generating system} for $S$. Numerical semigroups have a unique minimal generating system, and it has finitely many elements. The cardinality of this set is the \emph{embedding dimension} of the numerical semigroup (see for instance \cite[Chapter 1]{libro}).

If $S=\langle a_1,\ldots,a_p\rangle$, (with $a_1<\dots<a_p$), the number $a_1$ is called the \emph{multiplicity} of $S$. Clearly, the multiplicity is the least positive integer in $S$.

Let $s$ be a nonzero element of $S$. The \emph{Ap\'ery set} of $s$ in $S$ is the set $\Ap(S,s)=\{x\in S \mid x-s\notin S\}$. According to \cite[Lemma 1.4]{libro}, $\Ap(S,n)=\{0=w(0),w(1),\ldots,w(n-1)\}$, where $w(i)$ is the least element of $S$ congruent with $i$ modulo $n$, for all $i\in\{0,\ldots,n-1\}$.

The study of these invariants is an interesting topic of research, both from the pure semigroupist point of view or due to their counterpart in curves or algebro-geometry codes (see for instance \cite{ns-app} and the references therein). In this paper we focus in the study of the tree of numerical semigroups restricted to a fixed multiplicity, and we give some insight on how this tree is for multiplicities three, four and five. We use linear integer programming to study the growth of the number of numerical semigroups with low multiplicity as the genus increases. It was conjectured in \cite{conj-maria} that the number of numerical semigroups with given genus grows as the Fibonacci sequence. This conjecture was asymptotically proved in \cite{zhai}. In particular, the main result in \cite{zhai} implies that for a given genus $g$ large enough, there are more numerical semigroups with genus $g+1$ than those of genus $g$. However, there is no known lower bound when this occurs, and even if this bound where known and bigger than 68, we would not be able to answer the natural question: are there more numerical semigroups of genus $g+1$ than numerical semigroups of genus $g$? This is because we only know the number of numerical semigroups with genus $g$ for $g\le 67$ (see \cite{f-h}).

%For example, in \cite{kaplan} the author uses the genus and the Ap\'ery set for studying the Wilf's conjecture which is still open. Actually, in \cite{extensionWilf} this conjecture is proved for some families of monoids checking their multiplicity. Also, in \cite{multiplicity34} numerical semigroups with low multiplicity are studied. In \cite{shalom} a problem similar to the one developed in this work is showed.

The contents of this work is organized as follows. In Section \ref{sec:pre}, we show some previous results and some remarks that will be used later in the paper. In Section \ref{sec:mul-3}, we study semigroups of multiplicity three, and characterize their number of children depending on their Frobenius number and genus. In Section \ref{sec:mul-4}, we make use of the Kunz polytope for multiplicity four and the equation of the hyper-plane corresponding to semigroups with fixed genus $g$. Then we compute the number of integer points in this intersection, and we obtain that the number of numerical semigroups with multiplicity four and genus $g$ is an increasing function in $g$. We show that there are exactly the same number of numerical semigroups of genus $g$ and multiplicity four than unordered partitions of the integer $g+6$ in three elements such that the $i^{th}$ element in the partition is not $i$. In Section \ref{sec:mul-5}, we show that the number of numerical semigroups with multiplicity five and genus $g$ is also an increasing function in $g$.

\section{Preliminaries}\label{sec:pre}

Let $m$ be a positive integer and $g$ be a nonnegative integer. We denote by $\mathcal{S}(m,g)$ the set of numerical semigroups with multiplicity $m$ and genus $g$.

Let $S$ be a numerical semigroup with multiplcity $m$. Then the Ap\'ery set of $m$ in $S$ has $m$ elements. If, as above, $w(i)$ denotes the element in $\Ap(S,m)$ congruent with $i$ modulo $m$, then there exists a positive integer $k_i$ such that $w(i)=k_i m+i$. The tuple $(k_1,\ldots,k_{m-1})$ is known as the \emph{Kunz coordinates} of $S$. It is well known that the set of Kunz coordinates of numerical semigroups with multiplicity $m$ is precisely the set of nonnegative integer solutions of 
\[
\begin{array}{cc}
x_i\geq 1, & \hbox{for all } i\in \{1,\ldots,m-1\},\\
x_i+x_j-x_{i+j}\geq 0, & \hbox{for all } 1\leq i\leq j\leq m-1, i+j\leq m-1,\\
x_i+x_j-x_{i+j- m}\geq -1, &\hbox{for all } 1\leq i\leq j\leq m-1, i+j> m.
\end{array}
\]
These equations define a polytope known as the \emph{Kunz polytope} associated to $m$. In addition, if $S$ has genus $g$, then $k_1 + \dots+k_{m-1}=g$ (see \cite{kunz}). Thus the set $\mathcal{S}(m,g)$ is in one to one correspondence with the set of nonnegative integer solutions of 
\[
\begin{array}{cc}
x_i\geq 1, & \hbox{for all } i\in \{1,\ldots,m-1\},\\
x_i+x_j-x_{i+j}\geq 0, & \hbox{for all } 1\leq i\leq j\leq m-1, i+j\leq m-1,\\
x_i+x_j-x_{i+j-m}\geq -1, &\hbox{for all } 1\leq i\leq j\leq m-1, i+j> m,\\
x_1 + \dots+x_{m-1}=g.
\end{array}
\]
We can use any linear integer software to solve this system of equations and inequalities, and thus we have a procedure to compute the whole set $\mathcal{S}(m,g)$ for $m$ and $g$ fixed. For instance we can use the \texttt{GAP} \cite{gap} package \texttt{NormalizInterface} \cite{normalizinterface}, which is one of the many existing interfaces to \texttt{Normaliz} \cite{normaliz}, to compute this set of integer points. The package \texttt{NumericalSgps} \cite{numericalsgps} already has a function that computes the inequalities defining the Kunz polytope. 

\begin{verbatim}
gap> LoadPackage("normaliz");
gap> LoadPackage("num");
gap> eq:=KunzPolytope(4);
[ [ 1, 0, 0, -1 ], [ 0, 1, 0, -1 ], [ 0, 0, 1, -1 ], [ 2, -1, 0, 0 ], 
  [ 1, 1, -1, 0 ], [ 1, 1, -1, 0 ], [ -1, 1, 1, 1 ], [ -1, 1, 1, 1 ], 
  [ 0, -1, 2, 1 ] ]
gap> cone:=NmzCone(["inhom_inequalities",eq,
    "inhom_equations",[[1,1,1,-4]]]);
<a Normaliz cone>
gap> sol:=NmzModuleGenerators(cone);
[ [ 1, 1, 2, 1 ], [ 1, 2, 1, 1 ], [ 2, 1, 1, 1 ] ]
\end{verbatim}
We can alternatively compute these coordinates with \texttt{numericalsgps} in the following way. For high genus, this approach is not recommended.
\begin{verbatim}
gap> l:=NumericalSemigroupsWithGenus(4);;
gap> l4:=Filtered(l, s->Multiplicity(s)=4);;
gap> List(l4,KunzCoordinatesOfNumericalSemigroup);
[ [ 2, 1, 1 ], [ 1, 2, 1 ], [ 1, 1, 2 ] ]    
\end{verbatim}
For instance, if we want to compute the Kunz coordinates of all the elements in $\mathcal{S}(4,204)$ we may proceed as follows.
\begin{verbatim}
gap> cone:=NmzCone(["inhom_inequalities",eq,
    "inhom_equations",[[1,1,1,-204]]]);
<a Normaliz cone>
gap> sol:=NmzModuleGenerators(cone);;
gap> Length(sol);
3570
\end{verbatim}
This computation takes 38 milliseconds. However, \verb+NumericalSemigroupsWithGenus(204)+ will not stop (recall that we already mentioned in the introduction that the number of numerical semigroups of genus $g$ is known for $g\le 67$).

%A numerical semigroup has \emph{maximal embedding dimension}, MED for short, if its multiplicity equals its embedding dimension.

The following result will be helpfull to study the monotony of the number of numerical semigroups with multiplicity four or five as a function in the genus.

\begin{lemma}\label{creciente}
Let $f_i:\mathbb{N}\rightarrow\mathbb{N}$ be maps with $0\leq i\leq n-1$, and let $f:\mathbb{N}\rightarrow\mathbb{N}$ be a map defined as $f(m)=f_i(m)$, if $i\equiv m\bmod n$. If $f(kn)\leq f(kn+1)\leq \dots\leq f((k+1)n)$, then the function $f$ is nondecreasing.
\end{lemma}
\begin{proof}
Let $x,y\in\mathbb{N}$ be such that $x\leq y$, $x=k_1n+i$ and $y=k_2n+j$, with $i$, $j$ non-negative integers. If $k_1=k_2$, then the assertion follows directly from the hypothesis. So assume that $k_1<k_2$. Observe that, by induction, $f((k_1+1)n)\le f(k_2 n)$, since for every nonnegative integer $k$, the inequality $f(kn)\le f((k+1)n)$ holds. Hence, $f(x)=f(k_1n+i)\leq f(k_1n+i+1)\leq \dots \le f((k_1+1)n)\leq f(k_2n)\le  \dots\leq f(k_2n+j)=f(y)$.
\end{proof}

Let $S_1$, $S_2$ be numerical semigroups such that $S_1=S_2\cup \mathrm{F}(S_2)$. Then $S_2$ is called a \emph{child} of $S_1$.

The \emph{tree of numerical semigroups} is a tree with set of vertices the set of numerical semigroups, rooted in $\mathbb{N}$, and $e=(S,T)$ is an edge of this tree if and only if $S$ is a child of $T$, and $S=T\setminus\{x\}$ with $x$ a minimal generator of $T$ larger than $\mathrm{F}(T)$. In \cite[Proposition 1]{arbolSg} it is proved that this definition is well founded, that is, the tree of numerical semigroups is indeed a tree. Figure \ref{fig:arb-5} shows the tree of numerical semigroups of genus up to five (this figured was rendered with \texttt{d3js} \cite{d3}). Observe that a node in the tree of numerical semigroups is a leaf if all its minimal generators are below the Frobenius number. 

\begin{figure}
    \centering
    \includegraphics[scale=.75]{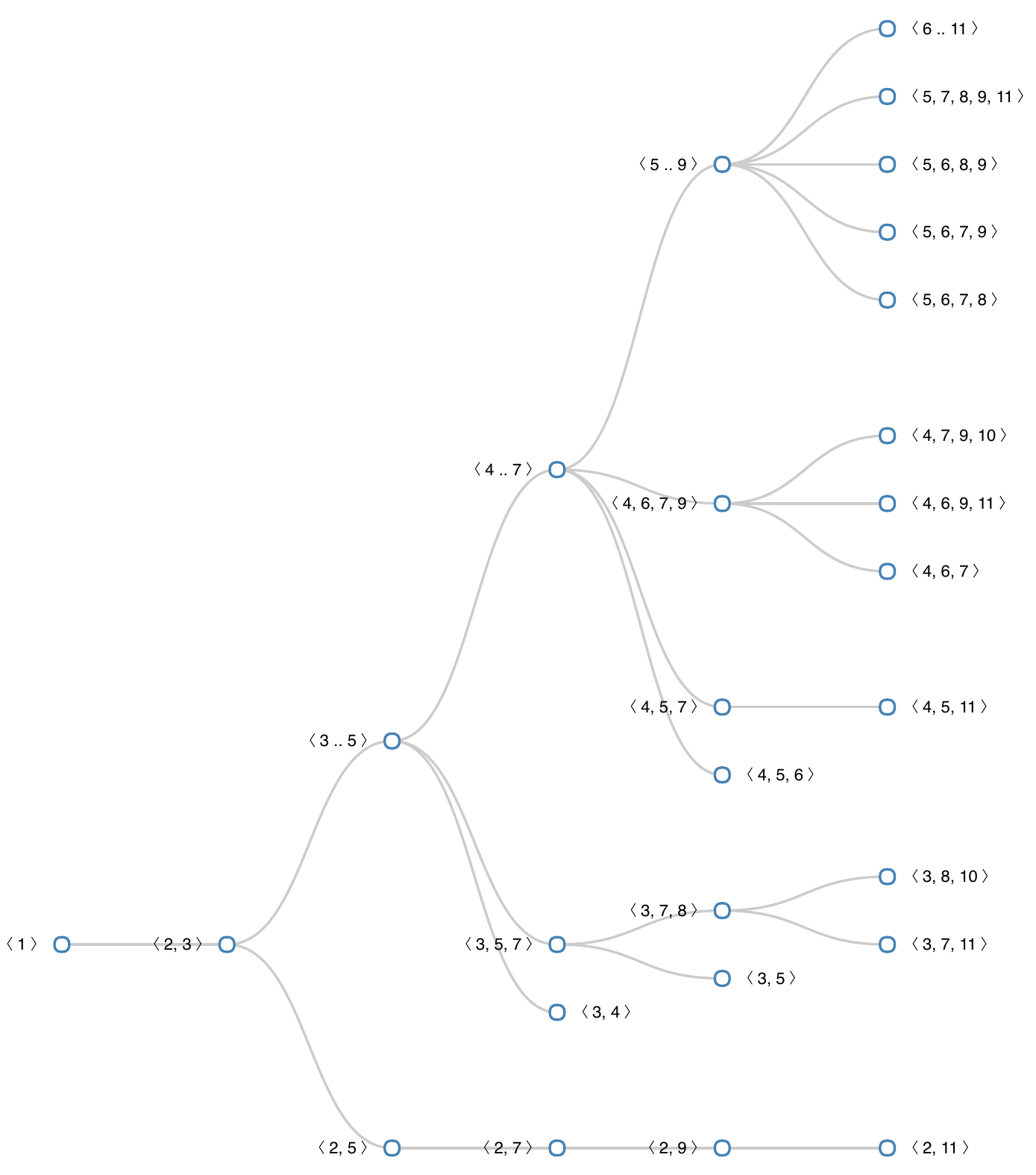}
    \caption{The tree of numerical semigroups up to genus 5}
    \label{fig:arb-5}
\end{figure}

Let $m$ be a positive integer. We can define the tree of numerical semigroups with multiplicity $m$ as the subtree of the tree of numerical semigroups with vertices those numerical semigroups with multiplicity $m$. 

\begin{remark}
Observe that there is at most one semigroup of embedding dimension two and multiplicity two for a fixed genus. Thus the tree of numerical semigroups with multiplicity two is simply a half line.
\end{remark}

A possible approach to solve the conjecture stating that number of numerical semigroups of genus $g$ does not decrease as $g$ increases would be to determine in each level (for each genus $g$) that there are less leaves than nodes with two or more children. This is what we have tried to achieve with low multiplicities. We were able to prove this for multiplicity three (for multiplicity one or two the problem is trivial). For multiplicity four the number of cases and subcases was too large, and thus we decided to use directly linear integer programming. 

\section{Tree of semigroups of multiplicity three}\label{sec:mul-3}

Numerical semigroups with multiplicity tree have been extensively studied in the literature. They have been parametrized in different ways. Next we recall some possibilities, one depending on the Frobenius number and genus, while the other relies on the remainder of the conductor modulo three. 

\begin{lemma}[{\cite[Lemma 6]{multiplicity34}}]\label{l:JC}
Let $S$ be a numerical semigroup with multiplicity three, Frobenius number $f$ and genus $g$. Then $\frac{f+1}{2}\leq g < \frac{2 f+3}{3}$.
\end{lemma}

\begin{theorem}[{\cite[Theorem 7]{multiplicity34}}]\label{t:JC}
Let $f$ be a positive integer greater than or equal to four that is not multiple of three. Let $g$ be a positive integer such that $\frac{f+1}{2}\leq g < \frac{2 f+3}{3}$. Then $S=\langle 3,3g-f,f+3 \rangle$ is a numerical semigroup with multiplicity three, Frobenius number $f$ and genus $g$.
\end{theorem}

\begin{remark}\label{r:cotaFrob} The inequalities
$\frac{f+1}{2}\leq g < \frac{2 f+3}{3}$ hold true if and only if $\frac{3 g-3}{2}< f \leq 2 g-1$.
\end{remark}

By using this two results, \cite[Corollary 10]{multiplicity34} gives a formula for the number of numerical semigroups with multiplicity three and genus $g$, $g$ a nonnegative integer: 
\[
\#\mathcal{S}(3,g)=g-\floor*{\frac{2g-1}3}.
\]
Note that this is equivalent to $\#\mathcal{S}(3,g)=\floor{g/3} +1$, which is nondecreasing. Also, in  \cite[Theorem~4.1]{counting}, it is shown that  $\#\mathcal{S}(3,g)=\ceil*{(g+1)/3}$, which is equivalent to the preceding expressions.

\begin{corollary}\label{c:nondecreasing-three}
The number of numerical semigroups with multiplicity three and genus $g$, with $g$ a nonnegative integer, is a nondecreasing function in $g$.
\end{corollary}

For multiplicity three we can distinguish which semigroups have children. We will show that if the numerical semigroup has embedding dimension two, then it cannot have children; while if it has embedding dimension three (maximal embedding dimension) it may have one or two children. 

\begin{lemma}\label{l:m3-e2}
Let $S$ be a numerical semigroup with multiplicity three and embedding dimension two. Then $S$ has no children.
\end{lemma}
\begin{proof}
If $S$ has embedding dimension two, then  $S=\langle 3,k\rangle$ for some integer $k$ with $k\ge 3$ and $\gcd(k,3)=1$. By Sylvester's formula, the Frobenius number of $S$ is $3k-3-k=2k-3$, which is greater than or equal to $k$. Thus $S$ has no minimal generators greater than the Frobenius number, and consequently it has no children.
\end{proof}

In order to see what happens in the maximal embedding dimension case, we will use the following result.

\begin{proposition}[{\cite[Proposition 3.1]{karakas}}]\label{t:karakas}
Let $S$ be a numerical semigroup with multiplicity three, conductor $c=3b + \bar{c}$, with $b\in\mathbb{N}$, and maximal embedding dimension.
	\begin{enumerate}[(i)]
	\item If $\bar{c}=0$, then $S=\langle 3, 3k_1+1, 3b+2 \rangle$, where $\ceil*{\frac{b+1}{2}}\leq k_1\leq b$,
	\item If $\bar{c}=2$, then $S=\langle 3, 3k_2+2, 3b+4 \rangle$, where $\ceil*{\frac{b+1}{2}}\leq k_1\leq b$.
	\end{enumerate}
\end{proposition}

Recall that in order to determine the number of children of a numerical semigroup in the tree of numerical semigroups, one has to know which generators are greater than the Frobenius number. Next, we give a better upper bound for the Frobenius number for the maximal embedding dimension case.

\begin{proposition}\label{p:bound-f-e3}
Let $S$ be a numerical semigroups with multiplicity three and maximal embedding dimension. Then $\mathrm{F}(S)\leq 2\mathrm{g}(S)-2$. 
\end{proposition}
\begin{proof}
Let $c$, $f$ and $g$ be the conductor, Frobenius number and genus of $S$, respectively. 

If $c\equiv 0\bmod 3$, by Proposition \ref{t:karakas}, $S=\langle 3, 3k_1+1, 3b+2 \rangle$, for some integer $k_1$ fulfilling $\ceil*{\frac{b+1}{2}}\leq k_1\leq b$. By Theorem \ref{t:JC}, $S=\langle 3,3 g-f,f+3 \rangle$. Then, we deduce that $b=\frac{f+1}{3}$ and $k_1=g-\frac{f+1}{3}$ where $\ceil*{\frac{f+4}{6}}\leq g-\frac{f+1}{3}\leq \frac{f+1}{3}$.

As $x<\ceil*{x}$, we have  $\frac{f+4}{6} \leq \ceil*{\frac{f+4}{6}} \leq g-\frac{f+1}{3}$, that is, $\frac{f+4}{6}\leq g-\frac{f+1}{3}$. So, $f\leq 2g-2$.

If $c\equiv 2\bmod 3$, by Proposition \ref{t:karakas}, $S=\langle 3, 3k_2+2, 3b+4 \rangle$ for some integer $k_1$ such that $\ceil*{\frac{b+1}{2}}\leq k_2\leq b$; by Theorem \ref{t:JC}, $S=\langle 3,3 g-f,f+3 \rangle$. Hence, $b=\frac{f-1}{3}$ and $k_2=g-\frac{f+2}{3}$, where $\ceil*{\frac{f+2}{6}}\leq g-\frac{f+2}{3}\leq \frac{f-1}{3}$.

Since $x<\ceil*{x}$, we obtain $\frac{f+2}{6} \leq \ceil*{\frac{f+2}{6}}\leq g-\frac{f+2}{3}$, that is, $\frac{f+2}{6}\leq g-\frac{f+2}{3}$. So, $f\leq 2 g-2$.
\end{proof}

% We give a condition for knowing if a MED numerical semigroup has one or two children.

% \begin{lemma}\label{l:doshijos}
% Let $G$ be a natural number such that $G\geq 2$. Then, there exist an unique numerical semigroup $S$ such that it has two children, multipicity three and genus $G$. 
% \end{lemma}
% \begin{proof}
% Trivially the embedding dimension of $S$ is three. Then $S=\langle 3,3G-F,F+3 \rangle$ by Theorem \ref{t:JC}. If $S$ has two children, then $F<3G-F$ so $2F<3G$ and $F<\frac{3G}{2}$. On the other hand, by Remark \ref{r:cotaFrob} $\frac{3G-3}{2}< F$. Therefore, if $S$ has two children, it verifies $\frac{3G-3}{2}< F<\frac{3G}{2}$ and there is only a natural number $F$ which verifies it.
% \end{proof}

With these results we can characterize numerical semigroups with embedding dimension three having no children, one child or two children, in the tree of numerical semigroups with embedding dimension three.

\begin{theorem}
Let $S$ be a numerical semigroup with multiplicity three. And let $\mathcal{T}_3$ be the tree of numerical semigroups of embedding dimension three. 	
\begin{enumerate}[1.]
    \item $S$ is a leaf in $\mathcal{T}_3$ if and only if $S$  has embedding dimension two.
	\item $S$ has one child in $\mathcal{T}_3$ if and only if $S$ has maximal embedding dimension and $\mathrm{F}(S)>\frac{3\mathrm{g}(S)}{2}$.
	\item $S$ has two children in $\mathcal{T}_3$ if and only if $S$ has maximal embedding dimension and $\mathrm{F}(S)<\frac{3\mathrm{g}(S)}{2}$.
\end{enumerate}
	Moreover, if we fix the genus $g\ge 5$, there is 
\begin{enumerate}[i.]
	\item a numerical semigroup with embedding dimension three and genus $g$ with two children in in $\mathcal{T}_3$; 
	\item a numerical semigroup with embedding dimension two and genus $g$ (and thus a leaf in $\mathcal{T}_3$) if and only if $g \bmod 3\in \{0,1\}$;
	\item one or more numerical semigroups with embedding dimension three and genus $g$ with one child in $\mathcal{T}_3$. 
\end{enumerate}
In particular,  
\[\#\mathcal{S}(3,g+1) = 
\begin{cases}
\#\mathcal{S}(3,g), & \hbox{if }\,g\bmod 3\in\{0,1\},\\
\#\mathcal{S}(3,g)+1, & \hbox{otherwise}.
\end{cases}
\]
\end{theorem}
\begin{proof}
We already know that $S$ has no children in the case $S$ has embedding dimension two (Lemma \ref{l:m3-e2}). Observe that if $S$ has embedding dimension two, $S=\langle 3,k\rangle$ for some integer $k$, and $k$ is determined by the Frobenius number of $S$ (by Sylsvester's formula). As $S$ is symmetric, the genus of $S$ equals $(\mathrm{F}(S)+1)/2$. Hence $k$ is determined by $\mathrm{g}(S)$. This means that for a fixed genus $g$ there is at most one numerical semigroup with multiplicity two and genus $g$. Sylverster's formula yields $\mathrm{F}(S)=3k-3-k$, whence $k=(\mathrm{F}(S)+1)/2+1 = g(S)+1>3$. Thus there will be only a semigroup with this conditions if $g>2$ and $g+1$ is not a multiple of $3$. That is, when $g>2$ and either $g\equiv 0 \bmod 3$ or $g\equiv 1 \bmod 3$.

By Lemma~\ref{l:JC} and Theorem~\ref{t:JC}, we also know that $S=\langle 3,3g-f, f+3\rangle$ with $f=\mathrm{F}(S)$ and $g=\mathrm{g}(S)$, and that $3 g-f<f+3$. Hence in the case $f<3 g-f$, $S$ has two children, and in the case $f>3 g-f$, the semigroup $S$ will have one child.  

By Remark~\ref{r:cotaFrob}, $\frac{3 g-3}{2}< f$. Therefore, if $S$ has two children, then $\frac{3 g-3}{2}< f<\frac{3 g}{2}$, and there is at most an integer $f$ that is not a multiple of three fulfilling these conditions. If $g\equiv 1 \bmod 2$, then $f=\frac{3 g-1}2$; while for $g\equiv 0 \bmod 2$, we have $f=\frac{3 g-2}2$.

If $f>3 g-f$, by Remark~\ref{r:cotaFrob} and Proposition~\ref{p:bound-f-e3}, we deduce $3 g/2<f\le 2 g-2$. Notice that $3 g/2< 2 g-2$ if and only if $g>4$. For $g$ even, that is, $g=2k$ for some positive integer $k$, we have $3k<F\le 4 k-2$, which makes $k-2$ possible integers. Observe that $3k+1$ is a possible choice of $f$ that is not a multiple of three. For $g=2 k+1$, we have $3 k+3/2<f\le 4 k$, whence $3 k+2\le f \le 4k$. Here $f=3k+2$ is a possible choice for $f$.
\end{proof}

\begin{remark}
Observe that with the above result we recover that  $\#\mathcal{S}(3,g)\leq \#\mathcal{S}(3,g+1)$. Indeed we know when  $\#\mathcal{S}(3,g)<\#\mathcal{S}(3,g+1)$.
\end{remark}

% \begin{corollary}
% Let $\mathcal{S}(3,G)$ be the set of numerical semigroups with multiplicity three and genus $G$. Then $\#\mathcal{S}(3,G)\leq \#\mathcal{S}(3,G+1)$.
% \end{corollary}

\begin{figure}
    \centering
    \includegraphics[scale=.75]{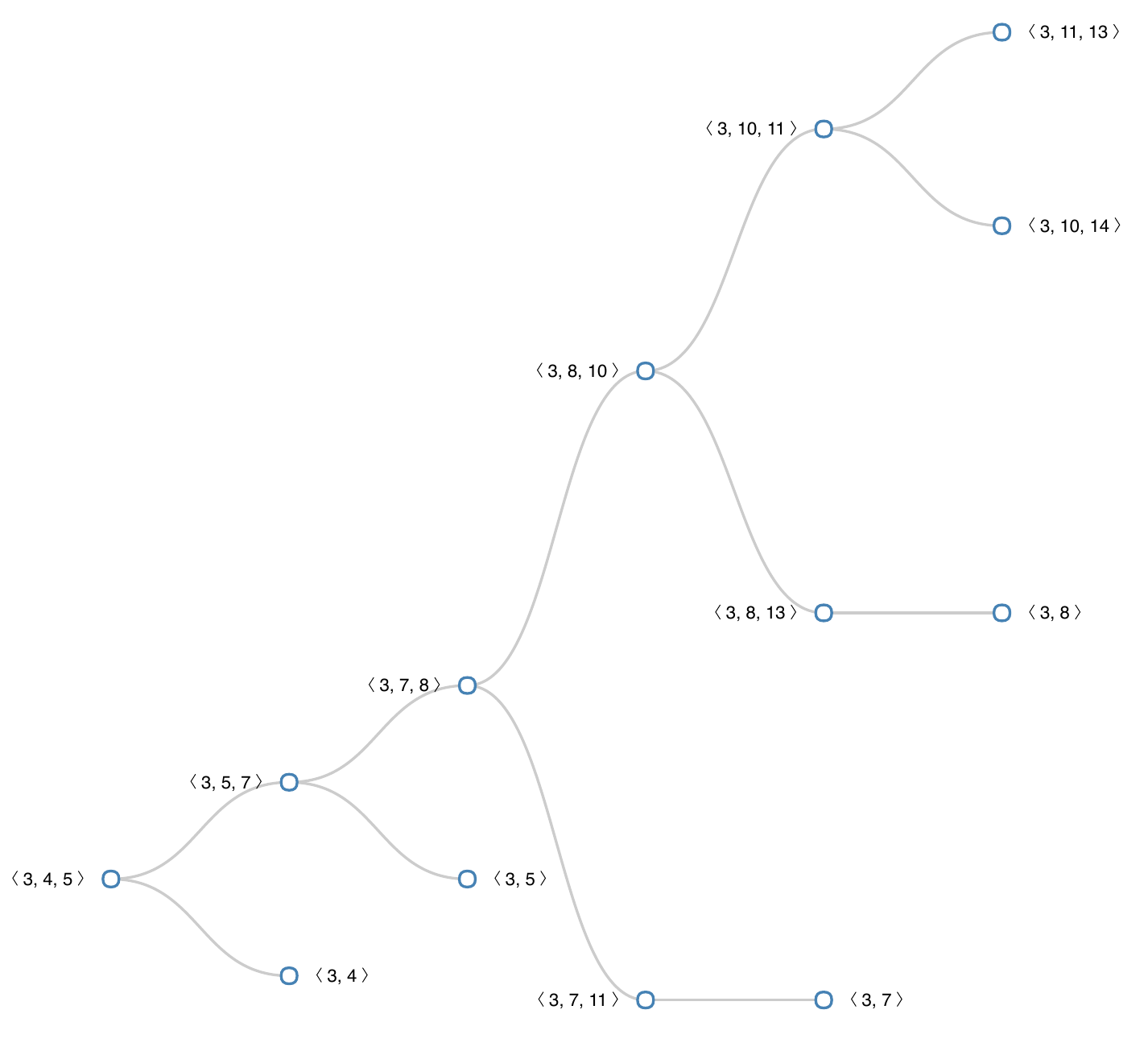}
    \caption{The tree of numerical semigroups with multiplicity three and genus up to seven}
    \label{fig:arb-m-3}
\end{figure}

\section{Tree of semigroups of multiplicity four}\label{sec:mul-4}

In this section we check that $\#\mathcal{S}(4,g)$ is a non decreasing map in $g$, that is, if $g_1\leq g_2$, then $\#\mathcal{S}(4,g_1)\leq\#\mathcal{S}(4,g_2)$. To this end, we will make use of the following result.

\begin{lemma}[{\cite[Lemma 2.2]{karakas}}]
Let $m\in\mathbb{N}$, $m>1$, and let $k=(k_1,\ldots,k_{m-1})\in\mathbb{N}^{m-1}$. Assume that
$$m^*=\max\big\{mk_i+i\mid\ 1\leq i\leq m-1\big\}=k_{i^*}m+i^*.$$
Then $k$ is the Kunz vector of a numerical semigroup with multiplicity $m$ %and major $M=k_{i^*}m+i^*$ 
if and only if
$$k_i+k_j-k_{(i+j)\bmod m}\geq
\left\{
\begin{array}{cl}
0, & \mbox{if } \; 1\leq i+j < m,\\
-1, & \mbox{if } \; i+j > m,
\end{array}
\right.$$
for all $i,j\in\{1,\ldots,m-1\}\setminus\{i^*\}$.
\end{lemma}

Now we study these equations with multiplicity $m=4$ and genus $g$. We distinguish three cases.

\begin{description}
	\item[Case 1] $m^*=4k_1+1$. Then we have the following equations:
	\begin{equation}
		\left\{
		\begin{array}{c}
		k_1+k_2+k_3=g, \\
		k_1,k_2,k_3\geq 1, \\
		k_2+k_3-k_1\geq -1, \\
		2k_3-k_2\geq -1, \\
		k_1>k_2, \\
		k_1>k_3. \\
		\end{array}
		\right.
	\end{equation}
	\item[Case 2] $m^*=4k_2+2$. In this case, the equations are:
	\begin{equation}
		\left\{
		\begin{array}{c}
		k_1+k_2+k_3=g, \\
		k_1,k_2,k_3\geq 1, \\
		2k_1-k_2\geq 0, \\
		2k_3-k_2\geq -1, \\
		k_2\geq k_1, \\
		k_2>k_3. \\
		\end{array}
		\right.
	\end{equation}
	\item[Case 3] $m^*=4k_3+3$. Under these conditions, the equations become:
	\begin{equation}
		\left\{
		\begin{array}{c}
		k_1+k_2+k_3=g, \\
		k_1,k_2,k_3\geq 1, \\
		k_1+k_2-k_3\geq 0, \\
		2k_1-k_2\geq 0, \\
		k_3\geq k_1, \\
		k_3\geq k_2. \\
		\end{array}
		\right.
	\end{equation}
\end{description}

By counting the number of solutions of these systems, using \texttt{barvinok} \cite{barvinok}, we get the following results.

\begin{proposition}\label{prop:count-4-f-1}
The number of numerical semigroups with multiplicity four and Frobenius number congruent with one modulo four is
\begin{itemize}
\item $1$, if $g=4$;
\item $\left(\frac{11}{4}-\frac{g}{2}\right) \left\lfloor \frac{g}{2} \right\rfloor + (g-8) \left\lfloor \frac{g+1}{3} \right\rfloor + \frac{g^2}{8} - \frac{9 g}{8}+9$, if $\ 5\leq g\leq 8$;
\item $5$, if $g=9$;
\item $\left\lfloor \frac{g+1}{5} \right\rfloor^2 - \frac{3}{2} \left\lfloor \frac{g+1}{3} \right\rfloor^2 - \frac{3}{2} \left\lfloor \frac{g+2}{6} \right\rfloor^2 - \left\lfloor \frac{2 g+4}{5} \right\rfloor^2 + \left(\frac{5}{4}-\frac{g}{2}\right) \left\lfloor \frac{g}{2} \right\rfloor + \left\lfloor \frac{g+1}{5} \right\rfloor + \left(g-\frac{1}{2}\right)
   \left\lfloor \frac{g+1}{3} \right\rfloor + \left(\left\lfloor \frac{g}{2} \right\rfloor -\frac{1}{2}\right) \left\lfloor \frac{g+2}{6} \right\rfloor + \left(-\left\lfloor \frac{g+1}{5} \right\rfloor + g + 1\right) \left\lfloor \frac{2 g+4}{5} \right\rfloor - \frac{g^2}{8} - \frac{7 g}{8}$, if $g \geq 10$.
\end{itemize}
\end{proposition}

The code used in Proposition \ref{prop:count-4-f-1} is
\begin{verbatim}
P := [g] -> { [x,y,z] : 
                x+y+z=g and x>=1 and y>=1 and z>=1 and
                y+z-x>=-1 and 2*z-y>=-1 and
                x>y and x>z
             };
card P;
\end{verbatim}

\begin{proposition}\label{prop:count-4-f-2}
The number of numerical semigroups with multiplicity four and Frobenius number congruent with two modulo four is
\begin{itemize}
\item 1, if $g=4$;
\item 1, if $g=5$;
\item $(7-g) \left\lfloor \frac{1}{5} (2 g+1) \right\rfloor - \frac{g^2}{4} + \frac{25 g}{4}-\frac{59}{2}$, if $\;6\leq g\leq 7$;
\item $-\left\lfloor \frac{g}{5}\right\rfloor^2 + \left\lfloor \frac{g}{4} \right\rfloor^2 - \frac{3}{2} \left\lfloor \frac{g}{3} \right\rfloor^2 + \left\lfloor \frac{2 g}{5} \right\rfloor^2 - \left\lfloor \frac{g+2}{5} \right\rfloor^2 + \left\lfloor \frac{g+2}{4} \right\rfloor^2 + \left\lfloor \frac{2 g+1}{5} \right\rfloor^2 - \left\lfloor \frac{g}{5} \right\rfloor + \left(g-\frac{3}{2}\right) \left\lfloor \frac{g}{3} \right\rfloor + (1-g) \left\lfloor \frac{2 g}{5} \right\rfloor +$ 

$\left\lfloor \frac{g}{4} \right\rfloor \left(1-\left\lfloor \frac{g}{2} \right\rfloor \right) + \left(\frac{g}{2}-\frac{1}{4}\right) \left\lfloor \frac{g}{2} \right\rfloor + \left\lfloor \frac{2 g}{5} \right\rfloor \left\lfloor \frac{g+2}{5} \right\rfloor - \left\lfloor \frac{g}{2} \right\rfloor \left\lfloor \frac{g+2}{4} \right\rfloor + \left(\left\lfloor \frac{g}{5} \right\rfloor - g + 1\right) \left\lfloor \frac{2 g+1}{5} \right\rfloor + \frac{g^2}{8} - \frac{g}{8}$, if $g\ge 8$.
\end{itemize}
\end{proposition}

The code used to obtain Proposition \ref{prop:count-4-f-2} is
\begin{verbatim}
P := [g] -> { [x,y,z] : 
                x+y+z=g and x>=1 and y>=1 and z>=1 and
                2*x-y>=0 and 2*z-y>=-1 and
                y>=x and y>z
             };
card P;
\end{verbatim}

\begin{proposition}\label{prop:count-4-f-3}
The number of numerical semigroups with multiplicity four and Frobenius number congruent with three modulo four is
\begin{itemize}
\item 1, if $g=3$;
\item $(-21/2 + 35/8 g - 3/8 g^2)$, if $\ 4\leq g\leq 5$;
\item 3, if $g=6$;
\item $\left\lfloor \frac{g}{5} \right\rfloor^2 - \left\lfloor \frac{2 g}{5} \right\rfloor^2 - \frac{3}{2} \left\lfloor \frac{g+2}{3} \right\rfloor^2 - \frac{3}{2}\left\lfloor \frac{g+5}{6} \right\rfloor^2 + \left\lfloor \frac{g}{5} \right\rfloor + \left(-\left\lfloor \frac{g}{5}\right\rfloor +g-1\right) \left\lfloor \frac{2 g}{5} \right\rfloor  + \left(\frac{g}{2}-\frac{3}{4}\right) \left\lfloor \frac{g}{2} \right\rfloor + \left(g + \frac{1}{2}\right) \left\lfloor \frac{g+2}{3} \right\rfloor +$ 

$\left(-\left\lfloor \frac{g}{2} \right\rfloor + g + \frac{1}{2}\right) \left\lfloor \frac{g+5}{6}\right\rfloor - \frac{5 g^2}{8} + \frac{5 g}{8}$, if $g\geq 7$.
\end{itemize}
\end{proposition}

The code for this case is

\begin{verbatim}
P := [g] -> { [x,y,z] : 
                x+y+z=g and x>=1 and y>=1 and z>=1 and
                x+y-z>=0 and 2*x-y>=0 and
                z>=x and z>=y
             };
card P;
\end{verbatim}

By Propositions~\ref{prop:count-4-f-1}, \ref{prop:count-4-f-2} and \ref{prop:count-4-f-3}, we obtain the next result.

\sloppy

\begin{theorem}\label{th:ngs-4}
Let $g$ be an integer greater than nine. Then $\#\mathcal{S}(4,g)=\left\lfloor \frac{g}{4} \right\rfloor^2 - \frac{3}{2} \left\lfloor \frac{g}{3} \right\rfloor^2 + \left\lfloor \frac{g+1}{5} \right\rfloor^2 - \frac{3}{2} \left\lfloor \frac{g+1}{3} \right\rfloor^2 - \frac{3}{2} \left\lfloor \frac{g+2}{6} \right\rfloor^2 - \left\lfloor \frac{g+2}{5} \right\rfloor ^2 + \left\lfloor \frac{g+2}{4} \right\rfloor^2 - \frac{3}{2} \left\lfloor \frac{g+2}{3} \right\rfloor^2 - \frac{3}{2} \left\lfloor \frac{g+5}{6} \right\rfloor^2 + \left\lfloor \frac{2 g+1}{5} \right\rfloor^2 - \left\lfloor \frac{2g+4}{5} \right\rfloor^2 + \left(g-\frac{3}{2} \right) \left\lfloor \frac{g}{3} \right\rfloor -\left\lfloor \frac{g}{5} \right\rfloor \left\lfloor \frac{2 g}{5} \right\rfloor + \left\lfloor \frac{g}{4}\right\rfloor \left(1-\left\lfloor \frac{g}{2} \right\rfloor \right) + \left(\frac{1}{4} + \frac{g}{2}\right) \left\lfloor \frac{g}{2}\right\rfloor + \left\lfloor \frac{g+1}{5} \right\rfloor + \left(g-\frac{1}{2}\right) \left\lfloor \frac{g+1}{3} \right\rfloor + \left(\left\lfloor \frac{g}{2} \right\rfloor - \frac{1}{2}\right) \left\lfloor \frac{g+2}{6} \right\rfloor + \left\lfloor \frac{2 g}{5} \right\rfloor \left\lfloor \frac{g+2}{5} \right\rfloor - \left\lfloor \frac{g}{2} \right\rfloor \left\lfloor \frac{g+2}{4} \right\rfloor + \left(g + \frac{1}{2}\right) \left\lfloor \frac{g+2}{3} \right\rfloor + \left(-\left\lfloor \frac{g}{2} \right\rfloor + g + \frac{1}{2}\right) \left\lfloor \frac{g+5}{6} \right\rfloor + \left(\left\lfloor \frac{g}{5} \right\rfloor - g + 1\right) \left\lfloor \frac{2 g+1}{5} \right\rfloor + \left(-\left\lfloor \frac{g+1}{5} \right\rfloor + g + 1\right) \left\lfloor \frac{2 g+4}{5} \right\rfloor -\frac{5 g^2}{8} - \frac{3 g}{8}$.
\end{theorem}

\fussy

It is easy to check that for $g\ge 10$, $\#\mathcal{S}(4,g)$ is a function in $g$ that verifies Lemma \ref{creciente} with $n=60$. The values of $\#\mathcal{S}(4,g)$ for $g\in\{0,\ldots,9\}$ can be obtained as follows.

\begin{verbatim}
gap> List([0..9], g->Length(Filtered(NumericalSemigroupsWithGenus(g), 
        s->Multiplicity(s)=4)));
[ 0, 0, 0, 1, 3, 4, 6, 7, 9, 11 ]
\end{verbatim}

Therefore we have the following consequence.

\begin{corollary}
Let $g$ be a nonnegative integer. Then $\#\mathcal{S}(4,g)\leq \#\mathcal{S}(4,g+1)$.
\end{corollary}

The functions given in Propositions \ref{prop:count-4-f-1} to \ref{prop:count-4-f-3} are not monotone, but as the last corollary says, the sum of the three is nondecreasing.

If one looks for the sequence $\#\mathcal{S}(4,g)$ with $g\in\mathbb{N}$ in The On-Line Encyclopedia of Integer Sequences \cite{enciclopedia-enteros}, one finds that it is related to an special kind of integer partitions. Next, we will explain this connection.

Let $n$ be a nonnegative integer and $k$ be a positive integer. An \emph{unordered partition} of $n$ in $k$ parts is a sequence $(p_1,\ldots,p_k)$ such that $p_1\le p_2\le \dots \le p_k$ and $p_1 + \dots+p_k=n$.

We show that the number of semigroups with multiplicity four and genus $g$ is precisely the number of partitions of $n=g+6$ into three parts such that every $i^{th}$ part is different from $i$. That is, the number of integer solutions of 
\[
\left\{\begin{array}{rcl}
    x+y+z &   =  & n, \\
        x & \geq & 2, \\
        y & \geq & 3, \\
        z & \geq & 4, \\
        y & \geq & x, \\
        z & \geq & y. \\
\end{array}\right.
\]
Using again \texttt{barvinok} with
\begin{verbatim}
P := [n] -> { [x,y,z] : 
                x+y+z=n and x>=2 and y>=3 and z>=4
                and x<=y and y<=z
            };
card P;
\end{verbatim}
we obtain that the number of solutions is
\[
\left\{\begin{array}{cl}
 1, & \mbox{if }\, n=9,  \\
 (n-12) + \lfloor\frac{n}{2}\rfloor, & \mbox{if }\, 10\leq n\leq 12,\\
 \frac{3}{2} \left\lfloor \frac{n+2}{3} \right\rfloor^2 + \frac{1}{2} \left\lfloor \frac{n}{2} \right\rfloor + \left(-n-\frac{3}{2}\right) \left\lfloor \frac{n+2}{3} \right\rfloor + \frac{n^2}{4} - \frac{n}{4}, & \mbox{if }\, n\geq 13.
\end{array}\right. 
\]

Therefore, if $n\geq 13$, the cardinal of the set of the solutions can be seen as a function $f(n)=f_i(n)$ with $n\equiv i \bmod 6$ where the functions $f_i(n)$ are: $f_0(n)=-3 n + 3 n^2$, $f_1(n)=-1 - 2 n + 3 n^2$, $f_2(n)=-1 - n + 3 n^2$, $f_3(n)=-1 + 3 n^2$, $f_4(n)=-1 + n + 3 n^2$ and $f_5(n)=-1 + 2 n + 3 n^2$. If we use these auxiliary functions together with the splitting in 60 functions of the expression in Theorem \ref{th:ngs-4}, the computation $f(n)-\#\mathcal{S}(4,n-6)$ yields 0. Hence, we get the following corollary.

\begin{corollary}
The number of semigroups with multiplicity four and genus $g$ is the same that the number of unordered partitions of $g+6$ into three parts such that every $i^{th}$ part is different from $i$.
\end{corollary}

%%%%%%%%%%%%%%%%%%%%%%%%%%%%%%%%%%%%%%%%%%%%%%%%%%%%%%%%%%%%%%%%%%%%%%%%%%%%%%%%%%%%%%%%%%%%%%%%%%%%%%%%%

\section{Other multiplicities} \label{sec:mul-5}

We can repeat the same procedure for other multiplicities. For example, for multiplicity five we get the following result.

\setlength\emergencystretch{\hsize}\hbadness=10000

\begin{proposition}\label{th:ngs-mult5}
If $g$ is an non negative number greater than 13, then $\#\mathcal{S}(5,g)$ equals
{\small $\frac{1215 g^3}{124}-\frac{1}{4} \left\lfloor \frac{g}{2}\right\rfloor  g^2-\frac{1771
   g^2}{124}-\left\lfloor \frac{1}{5} (2 g+2)\right\rfloor  g + \frac{201 g}{31} + \frac{3 \left\lfloor
   \frac{g}{7}\right\rfloor ^3}{2}-\frac{8 \left\lfloor \frac{g}{4}\right\rfloor ^3}{3}+3
   \left\lfloor \frac{2 g}{7}\right\rfloor ^3-\frac{13}{6} \left\lfloor \frac{3 g}{7}\right\rfloor
   ^3 + \frac{2}{3} \left\lfloor \frac{4 g}{9}\right\rfloor ^3-\frac{19}{22} \left\lfloor
   \frac{g+1}{7}\right\rfloor ^3-\frac{3}{2} \left\lfloor \frac{g+1}{6}\right\rfloor ^3-\frac{8}{3}
   \left\lfloor \frac{g+1}{4}\right\rfloor ^3-\frac{5}{3} \left\lfloor \frac{g+1}{3}\right\rfloor
   ^3 + \frac{7}{3} \left\lfloor \frac{g+2}{9}\right\rfloor ^3 + \frac{3}{2} \left\lfloor
   \frac{g+2}{7}\right\rfloor ^3 + \frac{3}{2} \left\lfloor \frac{g+2}{6}\right\rfloor ^3-\frac{2}{3}
   \left\lfloor \frac{g+2}{5}\right\rfloor ^3-\frac{8}{3} \left\lfloor \frac{g+2}{4}\right\rfloor
   ^3-\frac{1}{3} \left\lfloor \frac{g+2}{3}\right\rfloor ^3-\frac{17}{6} \left\lfloor
   \frac{g+3}{10}\right\rfloor ^3-\frac{7}{3} \left\lfloor \frac{g+3}{9}\right\rfloor ^3-\frac{7}{6}
   \left\lfloor \frac{g+3}{7}\right\rfloor ^3-\frac{4}{3} \left\lfloor \frac{g+3}{6}\right\rfloor
   ^3-\frac{1}{3} \left\lfloor \frac{g+3}{5}\right\rfloor ^3-\frac{8}{3} \left\lfloor
   \frac{g+3}{4}\right\rfloor ^3 + \frac{7}{3} \left\lfloor \frac{g+4}{9}\right\rfloor ^3+2
   \left\lfloor \frac{g+4}{7}\right\rfloor ^3-\frac{2}{3} \left\lfloor \frac{g+4}{6}\right\rfloor
   ^3-\frac{1}{3} \left\lfloor \frac{g+4}{5}\right\rfloor ^3-\frac{8}{3} \left\lfloor
   \frac{g+5}{10}\right\rfloor ^3-\frac{7}{3} \left\lfloor \frac{g+5}{9}\right\rfloor ^3-\frac{7}{6}
   \left\lfloor \frac{g+5}{7}\right\rfloor ^3+2 \left\lfloor \frac{g+5}{6}\right\rfloor
   ^3 + \frac{7}{3} \left\lfloor \frac{g+6}{9}\right\rfloor ^3-\frac{8}{3} \left\lfloor
   \frac{g+6}{7}\right\rfloor ^3-\frac{17}{6} \left\lfloor \frac{g+7}{10}\right\rfloor
   ^3-\frac{7}{3} \left\lfloor \frac{g+7}{9}\right\rfloor ^3-\frac{17}{6} \left\lfloor
   \frac{g+9}{10}\right\rfloor ^3 + \frac{5}{2} \left\lfloor \frac{1}{15} (2 g+1)\right\rfloor
   ^3 + \frac{4}{3} \left\lfloor \frac{1}{7} (2 g+1)\right\rfloor ^3 + \frac{427}{124} \left\lfloor
   \frac{1}{5} (2 g+1)\right\rfloor ^3 + \frac{17}{6} \left\lfloor \frac{1}{7} (2 g+2)\right\rfloor
   ^3-\frac{1}{3} \left\lfloor \frac{1}{5} (2 g+2)\right\rfloor ^3 + \frac{779}{264} \left\lfloor
   \frac{1}{7} (2 g+3)\right\rfloor ^3 + \frac{7}{3} \left\lfloor \frac{1}{7} (2 g+5)\right\rfloor
   ^3 + \frac{5}{2} \left\lfloor \frac{1}{15} (2 g+6)\right\rfloor ^3 + \frac{5}{6} \left\lfloor
   \frac{1}{7} (2 g+6)\right\rfloor ^3 + \frac{5}{2} \left\lfloor \frac{1}{15} (2 g+9)\right\rfloor
   ^3 + \frac{5}{2} \left\lfloor \frac{1}{15} (2 g+12)\right\rfloor ^3-\frac{1}{6} \left\lfloor
   \frac{1}{10} (3 g+1)\right\rfloor ^3-\frac{3}{2} \left\lfloor \frac{1}{7} (3 g+1)\right\rfloor
   ^3-\frac{3}{2} \left\lfloor \frac{1}{7} (3 g+2)\right\rfloor ^3 + \frac{361}{93} \left\lfloor
   \frac{1}{10} (3 g+3)\right\rfloor ^3 + \frac{2}{3} \left\lfloor \frac{1}{7} (3 g+3)\right\rfloor
   ^3-\frac{3191}{186} \left\lfloor \frac{1}{10} (3 g+4)\right\rfloor ^3 + \frac{2}{3} \left\lfloor
   \frac{1}{7} (3 g+4)\right\rfloor ^3-\frac{3}{2} \left\lfloor \frac{1}{7} (3 g+5)\right\rfloor
   ^3-\frac{3}{2} \left\lfloor \frac{1}{7} (3 g+6)\right\rfloor ^3-\frac{1}{6} \left\lfloor
   \frac{1}{10} (3 g+9)\right\rfloor ^3 + \frac{2}{3} \left\lfloor \frac{1}{9} (4 g+1)\right\rfloor
   ^3 + \frac{2}{3} \left\lfloor \frac{1}{9} (4 g+3)\right\rfloor ^3-\frac{4}{3} \left\lfloor
   \frac{1}{9} (4 g+4)\right\rfloor ^3-\frac{2}{3} \left\lfloor \frac{1}{9} (4 g+5)\right\rfloor
   ^3-\frac{2}{3} \left\lfloor \frac{1}{9} (4 g+7)\right\rfloor ^3 + \frac{2}{3} \left\lfloor
   \frac{1}{9} (4 g+8)\right\rfloor ^3+(5-g) \left\lfloor \frac{g}{9}\right\rfloor ^2+(g-5)
   \left\lfloor \frac{g}{9}\right\rfloor ^2+3 \left\lfloor \frac{g}{7}\right\rfloor ^2+(2 g-5)
   \left\lfloor \frac{g}{4}\right\rfloor ^2+2 \left(3-\frac{3 g}{2}\right) \left\lfloor \frac{2
   g}{7}\right\rfloor ^2+2 \left(\frac{1}{2}-\frac{g}{2}\right) \left\lfloor
   \frac{g}{3}\right\rfloor ^2+2 \left(\frac{g}{2}-\frac{1}{2}\right) \left\lfloor
   \frac{g}{3}\right\rfloor ^2 + \left(g-\frac{3}{2}\right) \left\lfloor \frac{3 g}{7}\right\rfloor
   ^2 + \left(\frac{3 g}{2}-3\right) \left\lfloor \frac{3 g}{7}\right\rfloor
   ^2 + \left(\frac{3}{2}-g\right) \left\lfloor \frac{4 g}{9}\right\rfloor ^2 + \left(\frac{2339
   g}{286}-\frac{614}{143}\right) \left\lfloor \frac{g+1}{7}\right\rfloor ^2 + \frac{1}{2}
   \left\lfloor \frac{g+1}{7}\right\rfloor ^2 + \left(\frac{3 g}{2}-\frac{3 \left\lfloor
   \frac{g}{2}\right\rfloor }{2}-\frac{3}{2}\right) \left\lfloor \frac{g+1}{6}\right\rfloor
   ^2 + \frac{1}{2} \left\lfloor \frac{g+1}{5}\right\rfloor ^2+(2 g-3) \left\lfloor
   \frac{g+1}{4}\right\rfloor ^2+3 \left(g-\frac{3 \left\lfloor \frac{g}{2}\right\rfloor
   }{2}-\frac{1}{2}\right) \left\lfloor \frac{g+1}{3}\right\rfloor ^2 + \left(-g + \frac{3 \left\lfloor
   \frac{g}{2}\right\rfloor }{2} + \frac{1}{2}\right) \left\lfloor \frac{g+1}{3}\right\rfloor ^2+(3-g)
   \left\lfloor \frac{g+2}{9}\right\rfloor ^2+(2 g-1) \left\lfloor \frac{g+2}{4}\right\rfloor
   ^2 + \left(-\frac{g}{2}-\frac{1}{2}\right) \left\lfloor \frac{g+2}{3}\right\rfloor
   ^2 + \left(\frac{g}{2} + \frac{1}{2}\right) \left\lfloor \frac{g+2}{3}\right\rfloor
   ^2 + \left(\frac{g}{2}+1\right) \left\lfloor \frac{g+2}{3}\right\rfloor
   ^2 + \left(g-\frac{5}{2}\right) \left\lfloor \frac{g+3}{9}\right\rfloor
   ^2 + \left(-g + \frac{\left\lfloor \frac{2 g}{7}\right\rfloor }{2} + \frac{1}{2}\right) \left\lfloor
   \frac{g+3}{7}\right\rfloor ^2 + \left(g + \frac{\left\lfloor \frac{2 g}{7}\right\rfloor }{2}-2\right)
   \left\lfloor \frac{g+3}{7}\right\rfloor ^2 + \left(\left\lfloor \frac{g}{3}\right\rfloor
    + \frac{1}{2}\right) \left\lfloor \frac{g+3}{6}\right\rfloor ^2 + \left(\frac{3 g}{2}-\frac{3
   \left\lfloor \frac{g}{2}\right\rfloor }{2}\right) \left\lfloor \frac{g+3}{6}\right\rfloor
   ^2 + \left(-\frac{3 g}{2} + \left\lfloor \frac{g}{3}\right\rfloor  + \frac{3 \left\lfloor
   \frac{g}{2}\right\rfloor }{2}+2\right) \left\lfloor \frac{g+3}{6}\right\rfloor
   ^2 + \left(\frac{g}{2} + \frac{3}{2}\right) \left\lfloor \frac{g+3}{5}\right\rfloor ^2+(2 g+1)
   \left\lfloor \frac{g+3}{4}\right\rfloor ^2 + \left(\frac{3 g}{2}-3 \left\lfloor
   \frac{g}{2}\right\rfloor  + \frac{3}{2}\right) \left\lfloor \frac{g+4}{6}\right\rfloor
   ^2 + \left(\left\lfloor \frac{g+1}{3}\right\rfloor  + \frac{1}{2}\right) \left\lfloor
   \frac{g+4}{6}\right\rfloor ^2 + \left(\frac{g}{2}+1\right) \left\lfloor \frac{g+4}{5}\right\rfloor
   ^2+(g-1) \left\lfloor \frac{g+5}{9}\right\rfloor ^2 + \left(-\left\lfloor \frac{g}{3}\right\rfloor
   -\frac{1}{2}\right) \left\lfloor \frac{g+5}{6}\right\rfloor ^2 + \left(-\left\lfloor
   \frac{g}{3}\right\rfloor -\left\lfloor \frac{g+2}{3}\right\rfloor \right) \left\lfloor
   \frac{g+5}{6}\right\rfloor ^2 + \left(\frac{\left\lfloor \frac{4 g}{9}\right\rfloor }{2}-g\right)
   \left\lfloor \frac{g+6}{9}\right\rfloor ^2+(g+1) \left\lfloor \frac{g+6}{7}\right\rfloor ^2+(g+2)
   \left\lfloor \frac{g+9}{10}\right\rfloor ^2+(5-g) \left\lfloor \frac{1}{15} (2 g+1)\right\rfloor
   ^2 + \left(-g + \frac{\left\lfloor \frac{g+1}{7}\right\rfloor }{2}+2\right) \left\lfloor \frac{1}{7}
   (2 g+1)\right\rfloor ^2 + \left(\frac{299 g}{124}-\frac{785}{248}\right) \left\lfloor \frac{1}{5}
   (2 g+1)\right\rfloor ^2 + \left(\frac{3}{2}-\frac{3 g}{2}\right) \left\lfloor \frac{1}{7} (2
   g+2)\right\rfloor ^2 + \left(-g-\frac{\left\lfloor \frac{3 g}{7}\right\rfloor }{2}+1\right)
   \left\lfloor \frac{1}{7} (2 g+2)\right\rfloor ^2 + \frac{1}{2} \left\lfloor
   \frac{g+1}{5}\right\rfloor  \left\lfloor \frac{1}{5} (2 g+2)\right\rfloor ^2 + \left\lfloor
   \frac{1}{5} (2 g+2)\right\rfloor ^2 + \left(-\frac{807 g}{1144} + \frac{145 \left\lfloor
   \frac{g+1}{7}\right\rfloor }{88} + \frac{5149}{2288}\right) \left\lfloor \frac{1}{7} (2
   g+3)\right\rfloor ^2+(2-g) \left\lfloor \frac{1}{15} (2 g+6)\right\rfloor ^2 + \left(\frac{3
   \left\lfloor \frac{g+6}{7}\right\rfloor }{2}-\frac{g}{2}\right) \left\lfloor \frac{1}{7} (2
   g+6)\right\rfloor ^2 + \left(\frac{1}{2}-g\right) \left\lfloor \frac{1}{15} (2 g+9)\right\rfloor
   ^2+(-g-1) \left\lfloor \frac{1}{15} (2 g+12)\right\rfloor ^2 + \left(\frac{g}{2}-\left\lfloor
   \frac{g+3}{5}\right\rfloor -\frac{1}{2}\right) \left\lfloor \frac{1}{10} (3 g+1)\right\rfloor
   ^2 + \left(2 g-\frac{5}{2}\right) \left\lfloor \frac{1}{7} (3 g+1)\right\rfloor ^2 + \left(\frac{3
   g}{2} + \frac{3}{2} \left\lfloor \frac{1}{7} (2 g+1)\right\rfloor -\frac{3}{2}\right) \left\lfloor
   \frac{1}{7} (3 g+2)\right\rfloor ^2 + \left(\frac{1157 g}{124} + \frac{2003}{124} \left\lfloor
   \frac{1}{5} (2 g+1)\right\rfloor -\frac{55}{4}\right) \left\lfloor \frac{1}{10} (3
   g+3)\right\rfloor ^2 + \left(-g-\frac{1}{2}\right) \left\lfloor \frac{1}{7} (3 g+3)\right\rfloor
   ^2 + \left(\frac{5163 g}{124}-\left\lfloor \frac{g}{2}\right\rfloor -\frac{1941}{124} \left\lfloor
   \frac{1}{5} (2 g+1)\right\rfloor -\frac{2489}{124}\right) \left\lfloor \frac{1}{10} (3
   g+4)\right\rfloor ^2 + \left(-g + \left\lfloor \frac{g+3}{7}\right\rfloor -\frac{1}{2}\right)
   \left\lfloor \frac{1}{7} (3 g+4)\right\rfloor ^2 + \left(\frac{3 g}{2} + \frac{3}{2} \left\lfloor
   \frac{1}{7} (2 g+3)\right\rfloor \right) \left\lfloor \frac{1}{7} (3 g+5)\right\rfloor ^2 + \left(2
   g + \frac{\left\lfloor \frac{g}{7}\right\rfloor }{2}-\frac{1}{2} \left\lfloor \frac{1}{7} (2
   g+6)\right\rfloor +1\right) \left\lfloor \frac{1}{7} (3 g+6)\right\rfloor
   ^2 + \left(\frac{g}{2}-\left\lfloor \frac{g+4}{5}\right\rfloor -\frac{3 \left\lfloor
   \frac{g+9}{10}\right\rfloor }{2}+1\right) \left\lfloor \frac{1}{10} (3 g+9)\right\rfloor
   ^2 + \left(\frac{3}{2}-g\right) \left\lfloor \frac{1}{9} (4 g+1)\right\rfloor
   ^2 + \left(-g + \left\lfloor \frac{g+2}{9}\right\rfloor  + \frac{3}{2}\right) \left\lfloor \frac{1}{9}
   (4 g+3)\right\rfloor ^2 + \left(g + \frac{1}{2}\right) \left\lfloor \frac{1}{9} (4 g+4)\right\rfloor
   ^2 + \left(g-\left\lfloor \frac{g}{9}\right\rfloor -\frac{1}{2}\right) \left\lfloor \frac{1}{9} (4
   g+4)\right\rfloor ^2 + \left(g-\left\lfloor \frac{g+5}{9}\right\rfloor  + \frac{1}{2}\right)
   \left\lfloor \frac{1}{9} (4 g+5)\right\rfloor ^2 + \left(g-\left\lfloor \frac{g+3}{9}\right\rfloor
    + \frac{1}{2}\right) \left\lfloor \frac{1}{9} (4 g+7)\right\rfloor ^2 + \left(-g + \left\lfloor
   \frac{g}{9}\right\rfloor  + \frac{1}{2}\right) \left\lfloor \frac{1}{9} (4 g+8)\right\rfloor
   ^2 + \frac{3 \left\lfloor \frac{g}{7}\right\rfloor }{2} + \left(-\frac{g^2}{2} + \frac{5
   g}{2}-\frac{7}{3}\right) \left\lfloor \frac{g}{4}\right\rfloor +2 \left(\frac{g^2}{2}-2
   g + \frac{3}{2}\right) \left\lfloor \frac{2 g}{7}\right\rfloor  + \left\lfloor
   \frac{g}{9}\right\rfloor  \left(-\frac{g^2}{2} + \frac{3 g}{2}-\frac{3 \left\lfloor
   \frac{g}{3}\right\rfloor ^2}{2} + \left(g-\frac{1}{2}\right) \left\lfloor \frac{g}{3}\right\rfloor
   -\frac{8}{3}\right) + \left\lfloor \frac{g}{9}\right\rfloor  \left(\frac{g^2}{2}-\frac{3
   g}{2} + \frac{3 \left\lfloor \frac{g}{3}\right\rfloor ^2}{2} + \left(\frac{1}{2}-g\right)
   \left\lfloor \frac{g}{3}\right\rfloor  + \frac{8}{3}\right) + \left(-\frac{g^2}{2} + \frac{3
   g}{2}-\frac{5}{6}\right) \left\lfloor \frac{3 g}{7}\right\rfloor  + \left(-\frac{g^2}{2}+2
   g-\frac{3}{2}\right) \left\lfloor \frac{3 g}{7}\right\rfloor  + \left(\frac{g^2}{2}-\frac{3
   g}{2} + \frac{5}{6}\right) \left\lfloor \frac{4 g}{9}\right\rfloor  + \left\lfloor
   \frac{g}{6}\right\rfloor ^2 \left(3-\frac{3 \left\lfloor \frac{g}{2}\right\rfloor }{2}\right)+2
   \left\lfloor \frac{g}{3}\right\rfloor  \left(-\frac{g^2}{4} + \frac{3 g}{4}-\frac{\left\lfloor
   \frac{g}{2}\right\rfloor }{2}-\frac{1}{3}\right)+2 \left\lfloor \frac{g}{3}\right\rfloor 
   \left(\frac{g^2}{4}-\frac{3 g}{4} + \frac{\left\lfloor \frac{g}{2}\right\rfloor
   }{2} + \frac{1}{3}\right) + \left(-\frac{g}{2}-\frac{1}{4}\right) \left\lfloor
   \frac{g}{2}\right\rfloor  + \left(-\frac{g^2}{4} + \frac{g}{4} + \frac{1}{4}\right) \left\lfloor
   \frac{g}{2}\right\rfloor  + \left(-\frac{g^2}{4} + \frac{g}{2} + \frac{1}{4}\right) \left\lfloor
   \frac{g}{2}\right\rfloor  + \left\lfloor \frac{g}{6}\right\rfloor ^2 \left(\frac{3 \left\lfloor
   \frac{g}{2}\right\rfloor }{2}-\frac{3}{2}\right) + \left\lfloor \frac{g}{6}\right\rfloor 
   \left(\frac{g^2}{8}-\frac{g}{8} + \left(\frac{5}{4}-\frac{g}{2}\right) \left\lfloor
   \frac{g}{2}\right\rfloor \right) + \left\lfloor \frac{g}{6}\right\rfloor 
   \left(-\frac{g^2}{8} + \frac{g}{8} + \left(\frac{g}{2}-\frac{9}{4}\right) \left\lfloor
   \frac{g}{2}\right\rfloor  + \frac{3}{2}\right) + \left(\frac{2625 g}{572}-\frac{7771}{3432}\right)
   \left\lfloor \frac{g+1}{7}\right\rfloor -\frac{\left\lfloor \frac{g+1}{7}\right\rfloor
   }{6} + \left(-\frac{3 g^2}{8} + \frac{7 g}{8} + \left(\frac{g}{2}-\frac{3}{4}\right) \left\lfloor
   \frac{g}{2}\right\rfloor \right) \left\lfloor \frac{g+1}{6}\right\rfloor
    + \left(\frac{g^2}{4}-\frac{g}{4} + \frac{\left\lfloor \frac{g}{2}\right\rfloor
   }{2} + \frac{1}{2}\right) \left\lfloor \frac{g+1}{5}\right\rfloor  + \left(-\frac{g^2}{2} + \frac{3
   g}{2}-\frac{1}{3}\right) \left\lfloor \frac{g+1}{4}\right\rfloor  + \left(\frac{3 g^2}{4}-\frac{3
   g}{4}+(1-g) \left\lfloor \frac{g}{2}\right\rfloor  + \frac{2}{3}\right) \left\lfloor
   \frac{g+1}{3}\right\rfloor +3 \left(-\frac{g^2}{2} + \frac{g}{2} + \left(g-\frac{1}{2}\right)
   \left\lfloor \frac{g}{2}\right\rfloor \right) \left\lfloor \frac{g+1}{3}\right\rfloor
    + \left(\frac{g^2}{2}-\frac{3 g}{2} + \frac{3 \left\lfloor \frac{g}{3}\right\rfloor
   ^2}{2} + \left(\frac{1}{2}-g\right) \left\lfloor \frac{g}{3}\right\rfloor  + \frac{2}{3}\right)
   \left\lfloor \frac{g+2}{9}\right\rfloor  + \left(-\frac{g^2}{2} + \frac{g}{2} + \frac{2}{3}\right)
   \left\lfloor \frac{g+2}{4}\right\rfloor  + \left\lfloor \frac{g+2}{6}\right\rfloor ^2
   \left(\frac{3}{2}-\left\lfloor \frac{g+2}{3}\right\rfloor
   \right) + \left(-\frac{g^2}{4}-\frac{g}{4}-\frac{\left\lfloor \frac{g}{2}\right\rfloor
   }{2}-\frac{2}{3}\right) \left\lfloor \frac{g+2}{3}\right\rfloor
    + \left(-\frac{g^2}{4}-\frac{g}{4}-\frac{\left\lfloor \frac{g}{2}\right\rfloor
   }{2}-\frac{1}{6}\right) \left\lfloor \frac{g+2}{3}\right\rfloor
    + \left(\frac{g^2}{4} + \frac{g}{4} + \frac{\left\lfloor \frac{g}{2}\right\rfloor
   }{2} + \frac{1}{6}\right) \left\lfloor \frac{g+2}{3}\right\rfloor  + \left\lfloor
   \frac{g+2}{6}\right\rfloor ^2 \left(-\frac{3 \left\lfloor \frac{g}{2}\right\rfloor
   }{2} + \left\lfloor \frac{g+2}{3}\right\rfloor +1\right) + \left\lfloor \frac{g+2}{6}\right\rfloor 
   \left(-\frac{g^2}{8} + \frac{g}{8}-\frac{1}{2} \left\lfloor \frac{g+2}{3}\right\rfloor
   ^2 + \left(\frac{g}{2}-\frac{5}{4}\right) \left\lfloor \frac{g}{2}\right\rfloor  + \frac{\left\lfloor
   \frac{g+2}{3}\right\rfloor }{2} + \frac{1}{6}\right) + \left\lfloor \frac{g+2}{6}\right\rfloor 
   \left(\frac{1}{2} \left\lfloor \frac{g+2}{3}\right\rfloor ^2-\frac{3 \left\lfloor
   \frac{g+2}{3}\right\rfloor }{2} + \frac{5}{6}\right) + \left(-\frac{g^2}{2} + \frac{g}{2}-\frac{3
   \left\lfloor \frac{g}{3}\right\rfloor ^2}{2} + \left(g-\frac{1}{2}\right) \left\lfloor
   \frac{g}{3}\right\rfloor -\frac{1}{6}\right) \left\lfloor \frac{g+3}{9}\right\rfloor
    + \left(\frac{3}{2} \left\lfloor \frac{2 g}{7}\right\rfloor ^2+(2-g) \left\lfloor \frac{2
   g}{7}\right\rfloor  + \frac{2}{3}\right) \left\lfloor \frac{g+3}{7}\right\rfloor
    + \left(\frac{g^2}{2}-\frac{g}{2} + \frac{3}{2} \left\lfloor \frac{2 g}{7}\right\rfloor ^2+(2-g)
   \left\lfloor \frac{2 g}{7}\right\rfloor \right) \left\lfloor \frac{g+3}{7}\right\rfloor
    + \left(-\frac{\left\lfloor \frac{g}{3}\right\rfloor ^2}{2}-\frac{\left\lfloor
   \frac{g}{3}\right\rfloor }{2} + \frac{1}{6}\right) \left\lfloor \frac{g+3}{6}\right\rfloor
    + \left(\frac{3 g^2}{8}-\frac{7 g}{8}-\frac{\left\lfloor \frac{g}{3}\right\rfloor
   ^2}{2}-\frac{\left\lfloor \frac{g}{3}\right\rfloor }{2} + \left(\frac{3}{4}-\frac{g}{2}\right)
   \left\lfloor \frac{g}{2}\right\rfloor  + \frac{1}{6}\right) \left\lfloor \frac{g+3}{6}\right\rfloor
    + \left(-\frac{3 g^2}{8}-\frac{g}{8} + \left(\frac{g}{2} + \frac{1}{4}\right) \left\lfloor
   \frac{g}{2}\right\rfloor  + \frac{1}{2}\right) \left\lfloor \frac{g+3}{6}\right\rfloor
    + \left(-\frac{g^2}{2}-\frac{g}{2} + \frac{2}{3}\right) \left\lfloor \frac{g+3}{4}\right\rfloor
    + \left(-\frac{g^2}{2}-g+(g+1) \left\lfloor \frac{g}{2}\right\rfloor -\frac{1}{2}\right)
   \left\lfloor \frac{g+4}{6}\right\rfloor  + \left(-\frac{1}{2} \left\lfloor
   \frac{g+1}{3}\right\rfloor ^2-\frac{\left\lfloor \frac{g+1}{3}\right\rfloor
   }{2} + \frac{1}{6}\right) \left\lfloor \frac{g+4}{6}\right\rfloor  + \left(-\frac{g^2}{4}-\frac{5
   g}{4} + \frac{\left\lfloor \frac{g}{2}\right\rfloor }{2}-\frac{2}{3}\right) \left\lfloor
   \frac{g+4}{5}\right\rfloor  + \left(-\frac{g^2}{2}-\frac{g}{2}-\frac{3}{2} \left\lfloor
   \frac{g+2}{3}\right\rfloor ^2 + \left(g + \frac{3}{2}\right) \left\lfloor \frac{g+2}{3}\right\rfloor
    + \frac{1}{3}\right) \left\lfloor \frac{g+5}{9}\right\rfloor  + \left(\frac{\left\lfloor
   \frac{g}{3}\right\rfloor ^2}{2} + \frac{\left\lfloor \frac{g}{3}\right\rfloor
   }{2}-\frac{1}{6}\right) \left\lfloor \frac{g+5}{6}\right\rfloor  + \left(\frac{\left\lfloor
   \frac{g}{3}\right\rfloor ^2}{2} + \frac{\left\lfloor \frac{g}{3}\right\rfloor }{2} + \frac{1}{2}
   \left\lfloor \frac{g+2}{3}\right\rfloor ^2-\frac{\left\lfloor \frac{g+2}{3}\right\rfloor
   }{2}-\frac{1}{3}\right) \left\lfloor \frac{g+5}{6}\right\rfloor
    + \left(\frac{g^2}{2}-\frac{g}{2} + \left\lfloor \frac{4 g}{9}\right\rfloor ^2 + \frac{3}{2}
   \left\lfloor \frac{g+1}{3}\right\rfloor ^2 + \left(\frac{3}{2}-g\right) \left\lfloor \frac{4
   g}{9}\right\rfloor  + \left(-g-\frac{1}{2}\right) \left\lfloor \frac{g+1}{3}\right\rfloor
   -\frac{1}{3}\right) \left\lfloor \frac{g+6}{9}\right\rfloor  + \frac{2 \left\lfloor
   \frac{g+6}{7}\right\rfloor }{3} + \left(-\frac{g^2}{4}-\frac{3 g}{4}-\frac{\left\lfloor
   \frac{g}{2}\right\rfloor }{2}-\frac{1}{6}\right) \left\lfloor \frac{g+9}{10}\right\rfloor
    + \left(\frac{g^2}{2}-\frac{3 g}{2}-\left\lfloor \frac{g+1}{7}\right\rfloor ^2-\frac{\left\lfloor
   \frac{g+1}{7}\right\rfloor }{2} + \frac{2}{3}\right) \left\lfloor \frac{1}{7} (2 g+1)\right\rfloor
    + \left\lfloor \frac{g+2}{5}\right\rfloor ^2 \left(\left\lfloor \frac{g}{2}\right\rfloor
    + \frac{1}{2} \left\lfloor \frac{1}{5} (2 g+1)\right\rfloor  + \frac{1}{2}\right) + \left(-\frac{3789
   g^2}{248} + \frac{4761 g}{248} + \left(\frac{g}{2} + \frac{1}{4}\right) \left\lfloor
   \frac{g}{2}\right\rfloor -\frac{1639}{248}\right) \left\lfloor \frac{1}{5} (2 g+1)\right\rfloor
    + \left\lfloor \frac{g+2}{5}\right\rfloor  \left(\left(-\left\lfloor \frac{g}{2}\right\rfloor
   -\frac{1}{2}\right) \left\lfloor \frac{1}{5} (2 g+1)\right\rfloor
    + \frac{1}{6}\right) + \left\lfloor \frac{1}{15} (2 g+1)\right\rfloor 
   \left(\frac{g^2}{2}-\left\lfloor \frac{1}{5} (2 g+1)\right\rfloor  g-\frac{3 g}{2} + \frac{3}{2}
   \left\lfloor \frac{g+1}{3}\right\rfloor ^2-\left\lfloor \frac{g+2}{5}\right\rfloor
   ^2 + \left\lfloor \frac{1}{5} (2 g+1)\right\rfloor ^2 + \left(\frac{1}{2}-g\right) \left\lfloor
   \frac{g+1}{3}\right\rfloor  + \left\lfloor \frac{g+2}{5}\right\rfloor  \left\lfloor \frac{1}{5} (2
   g+1)\right\rfloor  + \frac{5}{2}\right) + \left\lfloor \frac{g+4}{7}\right\rfloor ^2
   \left(-\left\lfloor \frac{1}{7} (2 g+2)\right\rfloor -\frac{1}{2}\right) + \left\lfloor
   \frac{g+4}{7}\right\rfloor ^2 \left(-\frac{g}{2}-\frac{\left\lfloor \frac{3 g}{7}\right\rfloor
   }{2} + \frac{1}{2} \left\lfloor \frac{1}{7} (2 g+2)\right\rfloor
   -\frac{1}{2}\right) + \left(\frac{g^2}{2}-g\right) \left\lfloor \frac{1}{7} (2 g+2)\right\rfloor
    + \left(\frac{g^2}{2}-\frac{3 g}{2} + \frac{3}{2} \left\lfloor \frac{3 g}{7}\right\rfloor ^2+(2-g)
   \left\lfloor \frac{3 g}{7}\right\rfloor  + \frac{2}{3}\right) \left\lfloor \frac{1}{7} (2
   g+2)\right\rfloor  + \left\lfloor \frac{g+4}{7}\right\rfloor  \left(\frac{1}{2} \left\lfloor
   \frac{1}{7} (2 g+2)\right\rfloor ^2 + \frac{1}{2} \left\lfloor \frac{1}{7} (2 g+2)\right\rfloor
   -\frac{1}{6}\right) + \left\lfloor \frac{g+4}{7}\right\rfloor  \left(-\left\lfloor \frac{3
   g}{7}\right\rfloor ^2 + \left(g-\frac{3}{2}\right) \left\lfloor \frac{3 g}{7}\right\rfloor
    + \frac{3}{2} \left\lfloor \frac{1}{7} (2 g+2)\right\rfloor ^2 + \frac{g}{2}+(1-g) \left\lfloor
   \frac{1}{7} (2 g+2)\right\rfloor -\frac{5}{6}\right) + \left\lfloor \frac{g+5}{10}\right\rfloor ^2
   \left(g + \left\lfloor \frac{g+1}{5}\right\rfloor -\left\lfloor \frac{1}{5} (2 g+2)\right\rfloor
   -1\right) + \left(\frac{g^2}{4}-\left\lfloor \frac{g+1}{5}\right\rfloor  g-\frac{g}{4} + \frac{1}{2}
   \left\lfloor \frac{g+1}{5}\right\rfloor ^2 + \frac{\left\lfloor \frac{g}{2}\right\rfloor
   }{2} + \frac{1}{3}\right) \left\lfloor \frac{1}{5} (2 g+2)\right\rfloor  + \left\lfloor
   \frac{g+3}{5}\right\rfloor  \left(-\frac{g^2}{4}-\frac{5 g}{4} + \frac{\left\lfloor
   \frac{g}{2}\right\rfloor }{2} + \left\lfloor \frac{1}{5} (2 g+2)\right\rfloor
   -\frac{1}{6}\right) + \left\lfloor \frac{g+5}{10}\right\rfloor  \left(\left\lfloor
   \frac{g+1}{5}\right\rfloor ^2+(1-g) \left\lfloor \frac{g+1}{5}\right\rfloor  + \left\lfloor
   \frac{1}{5} (2 g+2)\right\rfloor  \left\lfloor \frac{g+1}{5}\right\rfloor
    + \frac{2}{3}\right) + \left(g^2-\frac{3769 g}{1144}-\frac{1}{2} \left\lfloor
   \frac{g+1}{7}\right\rfloor ^2 + \left(\frac{1311}{1144}-\frac{2339 g}{286}\right) \left\lfloor
   \frac{g+1}{7}\right\rfloor -\frac{1381}{6864}\right) \left\lfloor \frac{1}{7} (2
   g+3)\right\rfloor  + \left(\frac{g^2}{2}-\frac{g}{2}-\left\lfloor \frac{g+1}{5}\right\rfloor
   ^2 + \frac{3}{2} \left\lfloor \frac{g+1}{3}\right\rfloor ^2 + \left\lfloor \frac{1}{5} (2
   g+2)\right\rfloor ^2-\left\lfloor \frac{g+1}{5}\right\rfloor  + \left(\frac{1}{2}-g\right)
   \left\lfloor \frac{g+1}{3}\right\rfloor  + \left(\left\lfloor \frac{g+1}{5}\right\rfloor -g\right)
   \left\lfloor \frac{1}{5} (2 g+2)\right\rfloor -\frac{1}{2}\right) \left\lfloor \frac{1}{15} (2
   g+6)\right\rfloor  + \left(\frac{1}{2} \left\lfloor \frac{g+6}{7}\right\rfloor ^2+(-g-1)
   \left\lfloor \frac{g+6}{7}\right\rfloor -\frac{g}{2}-\frac{5}{6}\right) \left\lfloor \frac{1}{7}
   (2 g+6)\right\rfloor  + \left(\frac{g^2}{2}-\left\lfloor \frac{2 g}{5}\right\rfloor 
   g-\frac{g}{2} + \left\lfloor \frac{2 g}{5}\right\rfloor ^2 + \frac{3}{2} \left\lfloor
   \frac{g+1}{3}\right\rfloor ^2-\left\lfloor \frac{g+4}{5}\right\rfloor
   ^2 + \left(\frac{1}{2}-g\right) \left\lfloor \frac{g+1}{3}\right\rfloor  + \left(\left\lfloor \frac{2
   g}{5}\right\rfloor +1\right) \left\lfloor \frac{g+4}{5}\right\rfloor -1\right) \left\lfloor
   \frac{1}{15} (2 g+9)\right\rfloor  + \left(\frac{3}{2} \left\lfloor \frac{g+1}{3}\right\rfloor
   ^2 + \left(\frac{1}{2}-g\right) \left\lfloor \frac{g+1}{3}\right\rfloor -\frac{3}{2} \left\lfloor
   \frac{g+3}{5}\right\rfloor ^2-\left\lfloor \frac{1}{5} (2 g+2)\right\rfloor ^2 + \left\lfloor
   \frac{g+3}{5}\right\rfloor  \left(g-\left\lfloor \frac{1}{5} (2 g+2)\right\rfloor
    + \frac{1}{2}\right)+g \left\lfloor \frac{1}{5} (2 g+2)\right\rfloor -\frac{1}{2}\right)
   \left\lfloor \frac{1}{15} (2 g+12)\right\rfloor  + \left\lfloor \frac{g+3}{10}\right\rfloor ^2
   \left(g-\frac{1}{2} \left\lfloor \frac{1}{10} (3 g+1)\right\rfloor
   -\frac{7}{2}\right) + \left(\frac{1}{2} \left\lfloor \frac{g+3}{5}\right\rfloor
   ^2 + \left(\left\lfloor \frac{1}{5} (2 g+2)\right\rfloor -\frac{1}{2}\right) \left\lfloor
   \frac{g+3}{5}\right\rfloor  + \left\lfloor \frac{1}{5} (2 g+2)\right\rfloor ^2 + \frac{g}{2}-g
   \left\lfloor \frac{1}{5} (2 g+2)\right\rfloor -\frac{1}{3}\right) \left\lfloor \frac{1}{10} (3
   g+1)\right\rfloor  + \left\lfloor \frac{g+3}{10}\right\rfloor  \left(-\frac{g^2}{4} + \frac{5
   g}{4}-\frac{3}{2} \left\lfloor \frac{1}{10} (3 g+1)\right\rfloor ^2-\frac{\left\lfloor
   \frac{g}{2}\right\rfloor }{2}+(g-1) \left\lfloor \frac{1}{10} (3 g+1)\right\rfloor
   -\frac{2}{3}\right) + \left\lfloor \frac{1}{7} (2 g+5)\right\rfloor ^2 \left(-2 g+3 \left\lfloor
   \frac{g+5}{7}\right\rfloor -\left\lfloor \frac{1}{7} (3 g+1)\right\rfloor
    + \frac{1}{2}\right)-\frac{3}{2} \left\lfloor \frac{g+2}{7}\right\rfloor ^2 \left\lfloor
   \frac{1}{7} (3 g+1)\right\rfloor  + \left(-g^2+2 g-1\right) \left\lfloor \frac{1}{7} (3
   g+1)\right\rfloor  + \left\lfloor \frac{1}{7} (2 g+5)\right\rfloor  \left(\frac{g^2}{2}-2
   \left\lfloor \frac{g+5}{7}\right\rfloor  g-\frac{g}{2} + 
   \left\lfloor \frac{g+5}{7}\right\rfloor
   ^2-\frac{1}{2} \left\lfloor \frac{1}{7} (3 g+1)\right\rfloor ^2 + \left(g + \frac{1}{2}\right)
   \left\lfloor \frac{1}{7} (3 g+1)\right\rfloor -\frac{5}{6}\right) + \left\lfloor
   \frac{g+2}{7}\right\rfloor  \left(\frac{1}{2} \left\lfloor \frac{1}{7} (3 g+1)\right\rfloor
   ^2-\frac{1}{2}\right) + \left(-\frac{g^2}{2}+g-\frac{1}{2} \left\lfloor \frac{1}{7} (2
   g+1)\right\rfloor ^2+(1-g) \left\lfloor \frac{1}{7} (2 g+1)\right\rfloor \right) \left\lfloor
   \frac{1}{7} (3 g+2)\right\rfloor  + \left(-\frac{759 g^2}{62} + \frac{896 g}{31} + \frac{821}{62}
   \left\lfloor \frac{1}{5} (2 g+1)\right\rfloor ^2 + \left(-g-\frac{106}{31}\right) \left\lfloor
   \frac{1}{5} (2 g+1)\right\rfloor -\frac{6559}{372}\right) \left\lfloor \frac{1}{10} (3
   g+3)\right\rfloor  + \left\lfloor \frac{g+5}{7}\right\rfloor ^2 \left(\frac{1}{2} \left\lfloor
   \frac{1}{7} (3 g+3)\right\rfloor -1\right) + 
   \left(\frac{g^2}{2} + \frac{g}{2}-\frac{1}{6}\right)
   \left\lfloor \frac{1}{7} (3 g+3)\right\rfloor  + 
   \left\lfloor \frac{g+5}{7}\right\rfloor 
   \left(\frac{g^2}{2} + \frac{g}{2} + \left\lfloor \frac{1}{7} (3 g+3)\right\rfloor
   ^2 + \left(-g-\frac{1}{2}\right) \left\lfloor \frac{1}{7} (3 g+3)\right\rfloor
    + \frac{1}{6}\right) + \left\lfloor \frac{g+7}{10}\right\rfloor ^2 \left(g-\frac{1}{2} \left\lfloor
   \frac{1}{10} (3 g+4)\right\rfloor \right) + \left(-\frac{3583 g^2}{124} + \left\lfloor
   \frac{g}{2}\right\rfloor  g + \frac{2489 g}{124}-\frac{485}{124} \left\lfloor \frac{1}{5} (2
   g+1)\right\rfloor ^2 + \left(\frac{2003 g}{62}-\left\lfloor \frac{g}{2}\right\rfloor
   -\frac{2489}{124}\right) \left\lfloor \frac{1}{5} (2 g+1)\right\rfloor  + \frac{821}{186}\right)
   \left\lfloor \frac{1}{10} (3 g+4)\right\rfloor  + \left\lfloor \frac{g+7}{10}\right\rfloor 
   \left(-\frac{g^2}{4} + \left\lfloor \frac{1}{10} (3 g+4)\right\rfloor  g + \frac{g}{4}-\frac{3}{2}
   \left\lfloor \frac{1}{10} (3 g+4)\right\rfloor ^2-\frac{\left\lfloor \frac{g}{2}\right\rfloor
   }{2} + \frac{5}{6}\right) + \left(\frac{g^2}{2} + \frac{g}{2} + \frac{1}{2} \left\lfloor
   \frac{g+3}{7}\right\rfloor ^2 + \left(-g-\frac{1}{2}\right) \left\lfloor \frac{g+3}{7}\right\rfloor
   -\frac{1}{6}\right) \left\lfloor \frac{1}{7} (3 g+4)\right\rfloor
    + \left(-\frac{g^2}{2}-\left\lfloor \frac{1}{7} (2 g+3)\right\rfloor  g-\frac{1}{2} \left\lfloor
   \frac{1}{7} (2 g+3)\right\rfloor ^2 + \frac{1}{2}\right) \left\lfloor \frac{1}{7} (3
   g+5)\right\rfloor  + \left(-g^2-g-\frac{3 \left\lfloor \frac{g}{7}\right\rfloor ^2}{2}-\left\lfloor
   \frac{1}{7} (2 g+6)\right\rfloor ^2-2 \left\lfloor \frac{g}{7}\right\rfloor
    + \left(g + \frac{3}{2}\right) \left\lfloor \frac{1}{7} (2 g+6)\right\rfloor -\frac{1}{2}\right)
   \left\lfloor \frac{1}{7} (3 g+6)\right\rfloor  + \left(-\frac{g^2}{2} + \left\lfloor \frac{2
   g}{5}\right\rfloor  g-g-\left\lfloor \frac{2 g}{5}\right\rfloor ^2-\frac{1}{2} \left\lfloor
   \frac{g+9}{10}\right\rfloor ^2 + \left(g-\left\lfloor \frac{2 g}{5}\right\rfloor +1\right)
   \left\lfloor \frac{g+4}{5}\right\rfloor +(g+2) \left\lfloor \frac{g+9}{10}\right\rfloor
   -\frac{5}{6}\right) \left\lfloor \frac{1}{10} (3 g+9)\right\rfloor  + \left\lfloor
   \frac{g+4}{9}\right\rfloor ^2 \left(-g + \frac{1}{2} \left\lfloor \frac{1}{9} (4 g+1)\right\rfloor
    + \frac{3}{2}\right) + \left(\frac{g^2}{2}-\frac{3 g}{2} + \frac{5}{6}\right) \left\lfloor \frac{1}{9}
   (4 g+1)\right\rfloor  + \left\lfloor \frac{g+4}{9}\right\rfloor 
   \left(\frac{g^2}{2}-\frac{g}{2} + \frac{3}{2} \left\lfloor \frac{g+2}{3}\right\rfloor
   ^2 + \left\lfloor \frac{1}{9} (4 g+1)\right\rfloor ^2 + \left(-g-\frac{3}{2}\right) \left\lfloor
   \frac{g+2}{3}\right\rfloor  + \left(\frac{3}{2}-g\right) \left\lfloor \frac{1}{9} (4
   g+1)\right\rfloor  + \frac{1}{6}\right) + \left(\frac{g^2}{2}-\frac{3 g}{2} + \frac{1}{2} \left\lfloor
   \frac{g+2}{9}\right\rfloor ^2 + \left(\frac{3}{2}-g\right) \left\lfloor \frac{g+2}{9}\right\rfloor
    + \frac{5}{6}\right) \left\lfloor \frac{1}{9} (4 g+3)\right\rfloor  + \left\lfloor
   \frac{g+7}{9}\right\rfloor ^2 \left(g-\frac{1}{2} \left\lfloor \frac{1}{9} (4 g+4)\right\rfloor
    + \frac{1}{2}\right) + \left(-\frac{g^2}{2}-\frac{g}{2} + \frac{1}{6}\right) \left\lfloor \frac{1}{9}
   (4 g+4)\right\rfloor  + \left(-\frac{g^2}{2} + \frac{g}{2}-\frac{\left\lfloor
   \frac{g}{9}\right\rfloor ^2}{2} + \left(g-\frac{1}{2}\right) \left\lfloor \frac{g}{9}\right\rfloor
    + \frac{1}{6}\right) \left\lfloor \frac{1}{9} (4 g+4)\right\rfloor  + \left\lfloor
   \frac{g+7}{9}\right\rfloor  \left(-\frac{g^2}{2}-\frac{g}{2}-\frac{3}{2} \left\lfloor
   \frac{g+2}{3}\right\rfloor ^2-\left\lfloor \frac{1}{9} (4 g+4)\right\rfloor
   ^2 + \left(g + \frac{1}{2}\right) \left\lfloor \frac{g+2}{3}\right\rfloor  + \left(g + \frac{1}{2}\right)
   \left\lfloor \frac{1}{9} (4 g+4)\right\rfloor
    + \frac{5}{6}\right) + \left(-\frac{g^2}{2}-\frac{g}{2}-\frac{1}{2} \left\lfloor
   \frac{g+5}{9}\right\rfloor ^2 + \left(g + \frac{1}{2}\right) \left\lfloor \frac{g+5}{9}\right\rfloor
    + \frac{1}{6}\right) \left\lfloor \frac{1}{9} (4 g+5)\right\rfloor
    + \left(-\frac{g^2}{2}-\frac{g}{2}-\frac{1}{2} \left\lfloor \frac{g+3}{9}\right\rfloor
   ^2 + \left(g + \frac{1}{2}\right) \left\lfloor \frac{g+3}{9}\right\rfloor  + \frac{1}{6}\right)
   \left\lfloor \frac{1}{9} (4 g+7)\right\rfloor  + \left(\frac{g^2}{2}-\frac{g}{2} + \frac{\left\lfloor
   \frac{g}{9}\right\rfloor ^2}{2} + \left(\frac{1}{2}-g\right) \left\lfloor \frac{g}{9}\right\rfloor
   -\frac{1}{6}\right) \left\lfloor \frac{1}{9} (4 g+8)\right\rfloor +1$.}
\end{proposition}

\fussy

\begin{corollary}
Let $\mathcal{S}(5,g)$ be the set of numerical semigroups with multiplicity five and genus $g$. Then $\#\mathcal{S}(5,g)\leq \#\mathcal{S}(5,g+1)$.
\end{corollary}

\begin{proof}
Applying Lemma \ref{creciente} with $n=79380$.
\end{proof}

These computations have been performed in \texttt{Wolfram Mathematica} \cite{mathematica} using the following code, where \texttt{mult5[g]} is the formula giving in Proposition~\ref{th:ngs-mult5}.

\begin{verbatim}
polynomials = Table[
  FullSimplify[mult5[79380 k + i], k \[Element] Integers], {i, 0, 79380}];
pol[k_] := Evaluate[polynomials];
dif = Expand[Table[
  pol[79380 k + i + 1][[i]] - pol[79380 k + i][[i]], {i, 1, 79381}]];
And@@Table[FullSimplify[dif[[i]] > 0, k > 0], {i, 1, 79381}]
\end{verbatim}


\begin{thebibliography}{20}

\bibitem{ns-app} 
\textsc{A. Assi, P.A. Garc\'ia-S\'anchez}, 
\newblock \emph{Numerical semigroups and applications}, 
\newblock RSME Springer series \textbf{1}, Springer, Switzerland, 2016.

\bibitem{counting} 
\textsc{V. Blanco, P.A. Garc\'{\i}a-S\'anchez, and J. Puerto},
\newblock \emph{Counting numerical semigroups with short generating functions},
\newblock Internat. J. Algebra Comput. \textbf{21}(7) (2011), 1217--1235.

\bibitem{d3}
\textsc{M. Bostock}, \emph{D3js Data-Driven Documents}, \url{https://d3js.org}.

\bibitem{conj-maria} 
\textsc{M. Bras-Amor\'os},
\newblock \emph{Fibonacci-like behavior of the number of numerical semigroups of a given genus}, 
\newblock Semigroup Forum \textbf{76} (2008), 379--384.

\bibitem{normaliz}
\textsc{W. Bruns, B. Ichim, T. R\"omer, R. Sieg, and C. S\"oger},
\newblock \emph{Normaliz. Algorithms for rational cones and affine monoids}. Available at \url{http://normaliz.uos.de}.

\bibitem{numericalsgps}
\textsc{M. Delgado, P.A. Garc\'ia-S\'anchez, and J. Morais},
\newblock \emph{NumericalSgps, A package  for  numerical  semigroups}, Version 1.1.5 dev (2017), (Refereed GAP package), \url{https://gap-packages.github.io/numericalsgps}.

% \bibitem{shalom}
% \textsc{S. Eliahou and J. Ram\'irez-Alfons\'in},
% \newblock \emph{On the number of numerical semigroups $\langle a,b \rangle$ of prime power genus}.
% \newblock Semigroup Forum \textbf{87}(1) (2013), 171--186. 

\bibitem{f-h}
\textsc{J. Fromentin, F. Hivert},
\newblock \emph{Exploring the tree of numerical semigroups},
\newblock Math. Comp. \textbf{85} (2016), 2553--2568.

\bibitem{gap}
\textsc{The GAP~Group},
\newblock \emph{GAP -- Groups, Algorithms, and Programming}, Version 4.8.8, (2017), \url{https://www.gap-system.org}.

% \bibitem{extensionWilf}
% \textsc{J.I. Garc\'ia-Garc\'ia, D. Mar\'in-Arag\'on, and A. Vigneron-Tenorio},
% \newblock \emph{An extension of Wilf's conjecture to affine semigroups}.
% \newblock To appear in Semigroup Forum. DOI: 10.1007/s00233-017-9906-1.

\bibitem{arbolSg}
\textsc{J.I. Garc\'ia-Garc\'ia, D. Mar\'in-Arag\'on, M.A. Moreno-Fr\'ias, J.C. Rosales, and A. Vigneron-Tenorio},
\newblock \emph{Frobenius number and minimum genus of numerical semigroups with fixed multiplicity and embedding dimension}.
\newblock arXiv:1712.05220v1 [math.AC] 14 Dec 2017.

\bibitem{normalizinterface}
\textsc{S. Gutsche, M. Horn, and C. S\"oger},
\newblock NormalizInterface, \emph{GAP wrapper for Normaliz},  Version 0.9.8 (2016), (GAP package), \url{https://gap-packages.github.io/NormalizInterface}.

% \bibitem{kaplan}. 
% \textsc{N. Kaplan},
% \newblock \emph{Counting numerical semigroups by genus and some cases of a questions of Wilf},
% \newblock J. Pure. Appl. Algebra \textbf{216}(5) (2012), 1016--1032.

\bibitem{karakas}
\textsc {H.I. Karakas},
\newblock \emph {Parametrizing numerical semigroups with multiplicity up to 5}.
\newblock To appear in Int. J. Algebra Comput. DOI: 10.1142/S0218196718500042.

\bibitem{mathematica}
\textsc{Wolfram Research, Inc.},
\newblock \emph{Mathematica Version 11.2},
\newblock Champaign, IL, 2017.

\bibitem{enciclopedia-enteros}
\textsc{OEIS}, \emph{The On-Line Encyclopedia of Integer Sequences}, \url{https://oeis.org}.

\bibitem{multiplicity34}
\textsc {J.C. Rosales},
\newblock \emph {Numerical semigroups with multiplicity three and four},
\newblock Semigroup Forum \textbf{71} (2005), 323--331.

\bibitem{libro}
\textsc{J.C. Rosales and P.A. Garc\'{\i}a-S\'anchez},
\newblock \emph{Numerical semigroups},
\newblock Developments in Mathematics, Vol. 20, Springer, New York, 2009.

\bibitem{kunz}
\textsc{J.C. Rosales, P.A. Garc\'ia-S\'anchez, J.I. Garc\'ia-Garc\'ia, and M.B. Branco},
\newblock \emph{Systems of inequalities and numerical semigroups},
\newblock J. Lond. Math. Soc. \textbf{65}(3) (2002), 611--623.

\bibitem{barvinok} 
\textsc{S. Verdoolaege}, 
\newblock \emph{Barvinok library}, \url{http://barvinok.gforge.inria.fr}.

\bibitem{zhai}
\textsc{A. Zhai},
\newblock \emph{Fibonacci-like growth of numerical semigroups with a given genus},
\newblock Semigroup Forum \textbf{86} (2013), 634--662.

\end{thebibliography}
\end{document}